\pdfoutput=1
\documentclass{amsart}
\usepackage[mathscr]{eucal}
\usepackage{amssymb}
\usepackage[usenames,dvipsnames]{xcolor} 
\usepackage[normalem]{ulem}
\usepackage{wasysym}
\usepackage{amsthm}
\usepackage{bbold}
\usepackage{comment}
\usepackage{enumitem}
\usepackage{amsmath}
\usepackage{amssymb}
\usepackage{tikz-cd}
\usetikzlibrary{arrows}
\usepackage{stmaryrd} 
\usepackage{etoolbox} 
\usepackage{microtype}
\usepackage{adjustbox}
\usepackage[unicode]{hyperref} 

\definecolor{dark-red}{rgb}{0.5,0.15,0.15}
\definecolor{dark-blue}{rgb}{0.15,0.15,0.6}
\definecolor{dark-green}{rgb}{0.15,0.6,0.15}

\hypersetup{
    colorlinks, linkcolor=OliveGreen,
    citecolor=OliveGreen, urlcolor=OliveGreen
}

\usepackage[nameinlink,capitalise,noabbrev]{cleveref}

\numberwithin{equation}{section}
\usepackage[all]{xy}
\xyoption{line}
\usepackage{graphicx}
\usepackage{mathtools}
\usepackage{caption}

\newtheorem{Thm}[equation]{Theorem}
\newtheorem*{Thm*}{Theorem}
\newtheorem*{MainThm*}{Main Theorem}
\newtheorem{Prop}[equation]{Proposition}
\newtheorem{Lem}[equation]{Lemma}
\newtheorem{Cor}[equation]{Corollary}

\newtheorem*{Que*}{Question}

\theoremstyle{remark}
\newtheorem{Def}[equation]{Definition}
\newtheorem{Ter}[equation]{Terminology}
\newtheorem{Not}[equation]{Notation}
\newtheorem{Exa}[equation]{Example}

\newtheorem{Cons}[equation]{Construction}

\newtheorem{Rem}[equation]{Remark}

\tikzset{
    labelrotatebelow/.style={anchor=north, rotate=90, inner sep=1.0mm}
}
\tikzset{
    labelrotateabove/.style={anchor=south, rotate=90, inner sep=1.0mm}
}
\SelectTips{cm}{10}

\newcommand{\nc}{\newcommand}
\nc{\dmo}{\DeclareMathOperator}

\nc{\Niko}[1]{{\color{Orange}#1}}
\nc{\Maxime}[1]{{\color{Green}#1}}
\nc{\Luca}[1]{{\color{Blue}#1}}

\usepackage{todonotes}

\nc{\overbar}[1]{\mkern 1.5mu\overline{\mkern-1.5mu#1\mkern-1.5mu}\mkern 1.5mu}

\nc{\kappaaux}{g}
\nc{\kappaCh}{{\kappaaux(\cat C_h)}}
\nc{\kappam}{{\kappaaux({\mathfrak m})}}
\nc{\kappaP}{\Gamma_{\cat P}\unit}
\nc{\kappaQ}{{\kappaaux(\cat Q)}}
\nc{\kappaCP}{{\kappaaux_{\cat C}(\cat P)}}
\nc{\kappaDP}{{\kappaaux_{\cat D}(\cat P)}}
\nc{\kappaCQ}{{\kappaaux_{\cat C}(\cat Q)}}
\nc{\kappaDQ}{{\kappaaux_{\cat D}(\cat Q)}}
\nc{\kappaphiB}{{\kappaaux(\phi(\cat B))}}
\nc{\kappaphiQ}{{\kappaaux(\varphi(\cat Q))}}

\dmo{\Sub}{Sub}
\dmo{\Proj}{Proj}
\dmo{\LMod}{LMod}
\dmo{\cell}{cell}
\nc{\SpEn}{\cat S_{E(n)}}
\nc{\SpEnf}{\cat S_n}
\nc{\Lcomp}{L^{\mathrm{com}}} 
\nc{\Ucomp}{U^{\mathrm{com}}}
\nc{\Loco}[1]{\Loc_{\otimes}\hspace{-0.3ex}\langle #1 \rangle}
\nc{\bbullet}{{\scriptscriptstyle\hspace{-1pt}\bullet}}
\nc{\bullett}{{\scriptscriptstyle\bullet}\hspace{-1pt}}
\nc{\LF}{L\hspace{-0.2ex}F}
\nc{\SpG}{\Sp^G}
\nc{\EG}{\bbE_G}
\nc{\DEG}{\Der(\EG)}
\nc{\DE}{\Der(\bbE)}
\nc{\Prst}{{\cat P}\mathrm{r^{st}}}
\nc{\Mack}[2]{\mathrm{Mack}_{#1}(#2)}
\nc{\SC}{S\cat C}
\dmo{\fin}{{fin}}
\dmo{\DM}{DM}
\dmo{\fp}{fp}
\nc{\DMQ}{\DM_Q}
\dmo{\DerKal}{DMack}
\dmo{\Der}{D}
\dmo{\DMot}{DMot}
\dmo{\rmH}{H}
\dmo{\piu}{\underline{\pi}}
\dmo{\Sphere}{\mathbb{S}}
\dmo{\Alg}{Alg}
\dmo{\CAlg}{CAlg}
\nc{\HA}{{\rmH \hspace{-0.2em}\bbA}}
\nc{\HZ}{{\rmH \hspace{-0.2em}\bbZ}}
\nc{\HZbar}{{\rmH \hspace{-0.2em}\underline{\bbZ}}}
\nc{\Fp}{{\bbF_{\hspace{-0.1em}p}}}
\nc{\HFp}{{\rmH \hspace{-0.15em}\bbF_{\hspace{-0.1em}p}}}
\nc{\DHZG}{\Der(\HZ_G)}
\nc{\DHZH}{\Der(\HZ_H)}
\nc{\DHZK}{\Der(\HZ_K)}
\nc{\DHZGN}{\Der(\HZ_{G/N})}
\nc{\DHZGG}{\Der(\HZ_{G/G})}
\nc{\DHZCp}{\Der(\HZ_{C_p})}
\nc{\DHZGprime}{\Der(\HZ_{G'})}
\nc{\DHZ}{\Der(\HZ)}
\nc{\mathfrakp}{\mathfrak{p}}
\nc{\mathfrakq}{\mathfrak{q}}
\nc{\mathfrakS}{\mathfrak{S}}
\nc{\mathfrakT}{\mathfrak{T}}
\nc{\Z}{\mathbb{Z}}
\nc{\SSG}{\text{sSet}_*^G}
\nc{\sSet}{\text{sSet}}

\dmo{\csupp}{csupp}
\dmo{\Con}{Conj}
\dmo{\Id}{Id}
\dmo{\Loc}{Loc}
\dmo{\rmK}{\textrm{\rm K}}
\dmo{\Spc}{Spc}
\dmo{\thick}{thick}
\nc{\thickt}[1]{\thick_\otimes\langle #1 \rangle}
\nc{\loct}[1]{\loc_\otimes\langle #1 \rangle}
\dmo{\cone}{cone}
\dmo{\End}{End}
\dmo{\Derperf}{D_{perf}}
\dmo{\Mor}{Mor}
\dmo{\Hom}{Hom}
\dmo{\id}{id}
\dmo{\incl}{incl}
\dmo{\Img}{Im}
\dmo{\im}{im}
\dmo{\Ker}{Ker}
\dmo{\ind}{ind}
\dmo{\Ind}{Ind}
\dmo{\CoInd}{coind}
\dmo{\res}{res}
\dmo{\infl}{infl}
\dmo{\Derqc}{D_{qc}}
\dmo{\triv}{triv}
\dmo{\sep}{sep}
\dmo{\Tel}{Tel} 
\dmo{\grMod}{grMod}%
\dmo{\Mod}{Mod}%
\dmo{\opname}{op}
\dmo{\SH}{SH}
\dmo{\smallb}{b}
\dmo{\Spec}{Spec}
\dmo{\supp}{supp}
\dmo{\Supp}{Supp}
\dmo{\cosupp}{cosupp}
\dmo{\Cosupp}{Cosupp}
\nc{\SHc}{{\SH^c}}
\nc{\SHp}{{\SH_{(p)}}}
\nc{\SHcp}{{\SH^c_{(p)}}}
\nc{\SHG}{\SH(G)}
\nc{\SHGp}{\SH(G)_{(p)}}
\nc{\SHGc}{\SHG^c}
\nc{\SHGcp}{\SHG^c_{(p)}}
\nc{\quadtext}[1]{\quad\textrm{#1}\quad}
\nc{\qquadtext}[1]{\qquad\textrm{#1}\qquad}
\nc{\adj}{\dashv}
\nc{\adjto}{\rightleftarrows}
\nc{\bbL}{\mathbb{L}}
\nc{\bbA}{\mathbb{A}}
\nc{\bbE}{\mathbb{E}}
\nc{\bbN}{\mathbb{N}}
\nc{\bbQ}{\mathbb{Q}}
\nc{\bbZ}{\mathbb{Z}}
\nc{\bbF}{\mathbb{F}}
\nc{\cat}[1]{\mathscr{#1}}
\nc{\ie}{{\sl i.e.}, }
\nc{\into}{\mathop{\rightarrowtail}}
\nc{\inv}{^{-1}}
\nc{\isoto}{\mathop{\overset{\sim}\to}}
\nc{\isotoo}{\mathop{\overset{\sim}\too}}
\nc{\onto}{\mathop{\twoheadrightarrow}}
\nc{\too}{\mathop{\longrightarrow}\limits}
\nc{\mapstoo}{\longmapsto}
\nc{\adh}[1]{\overline{#1}}
\nc{\adhpt}[1]{\adh{\{#1\}}}
\nc{\aka}{{a.\,k.\,a.}\ }
\nc{\calF}{\cat F}
\nc{\eg}{{\sl e.\,g.}}
\nc{\Homcat}[1]{\Hom_{\cat #1}}
\nc{\hook}{\hookrightarrow}
\nc{\ideal}[1]{\langle #1\rangle}
\nc{\ihom}{{\underline{\hom}}}
\nc{\iHom}{\underline{\mathrm{Hom}}}
\nc{\Mid}{\,\big|\,}
\nc{\MMod}{\,\text{-}\Mod}%
\nc{\GrMMod}{\,\text{-}\grMod}%
\nc{\op}{^{\opname}}
\nc{\oto}[1]{\overset{#1}\to}
\nc{\otoo}[1]{\overset{#1}{\,\too\,}}
\nc{\sminus}{\!\smallsetminus\!}
\nc{\poplus}[1]{^{\oplus #1}}%
\nc{\potimes}[1]{^{\otimes #1}}
\nc{\sbull}{{\scriptscriptstyle\bullet}}
\nc{\SET}[2]{\big\{\,#1\Mid#2\,\big\}}
\nc{\SpcK}{\Spc(\cat K)}
\nc{\then}{\Rightarrow}
\nc{\unit}{\mathbb{1}}
\nc{\xra}{\xrightarrow}
\nc{\phigeom}[1]{\widetilde{\Phi}^{#1}}
\dmo{\Oname}{O}
\dmo{\proper}{proper}
\dmo{\lenormal}{\unlhd}
\dmo{\lnormal}{\lhd}
\nc{\normal}{\trianglelefteq}
\nc{\Op}{\Oname^p}
\nc{\Oq}{\Oname^q}
\dmo{\Sp}{Sp}
\dmo{\Ho}{Ho}
\dmo{\Fin}{Fin}
\dmo{\add}{add}
\dmo{\Fun}{Fun}
\dmo{\Ext}{Ext}

\dmo{\CMon}{CMon}
\dmo{\BB}{\cat B}
\dmo{\CC}{\cat C}
\dmo{\DD}{\cat D}
\dmo{\MM}{\cat M}
\dmo{\NN}{\cat N}
\dmo{\OO}{\mathcal{O}}
\dmo{\Map}{Map}
\dmo{\Span}{Span}
\dmo{\N}{N}
\dmo{\Cat}{Cat}
\dmo{\colim}{colim}
\dmo{\hocolim}{hocolim}
\dmo{\Ch}{Ch}
\dmo{\A}{\mathbb{A}^{eff}}
\nc{\AGeff}{\mathbb{A}_G^{\mathrm{eff}}}
\nc{\BGeff}{\mathcal{B}_G^{\mathrm{eff}}}
\nc{\BG}{{\mathcal{B}_G}}
\nc{\NBGeff}{{\N}{\BGeff}}
\dmo{\Ab}{Ab}
\dmo{\Set}{Set}
\dmo{\ev}{ev}
\dmo{\Spcl}{Spcl}
\nc{\Funadd}{\Fun_{\add}}
\dmo{\proj}{proj}
\dmo{\cof}{cof}

\nc{\StModfin}{\mathrm{StMod}^{\mathrm{fin}}}
\nc{\PrL}{\mathrm{Pr}^{L}_{\mathrm{st}}}

\dmo{\Coideal}{Coideal}
\dmo{\gen}{gen}
\dmo{\Coloc}{Coloc}
\nc{\Coloco}[1]{\Coloc^{\iHom}\hspace{-0.3ex}\langle #1 \rangle}

\dmo{\dual}{dual}

\nc{\LOCO}{\mathcal{L}\mathrm{oc}_{\otimes}}
\nc{\COLOCO}{\mathcal{C}\mathrm{oloc}^{\iHom}}

\nc{\Perff}[1]{\mathrm{Perf}_{#1}}
\nc{\Modd}[1]{\mathrm{Mod}_{#1}}
\dmo{\Perf}{Perf}
\dmo{\tel}{tel}
\nc{\Lococat}[2]{\Loc^{#1}_{\otimes}\hspace{-0.3ex}\langle #2 \rangle}
\dmo{\rk}{rk}
\dmo{\StMod}{StMod}
\dmo{\stmod}{stmod}

\nc{\Ccpl}[1]{\CC^{#1\text{-}\mathrm{cpl}}}
\nc{\Ctors}[1]{\CC^{#1\text{-}\mathrm{tors}}}
\nc{\Cloc}[1]{\CC^{#1\text{-}\mathrm{loc}}}
\nc{\fib}[1]{\mathrm{fib}}
\nc{\tors}{\mathrm{tors}}
\nc{\cpl}{\mathrm{cpl}}
\nc{\loc}{\mathrm{loc}}

\newcommand{\K}{\mathcal K}

\nc{\QW}{W^{\text{Q}}}

\nc{\mT}{\kern-0.5em\mod\kern-0.1em\text{-}\cat{T}^c}
\nc{\mTc}{\kern-0.5em\mod\kern-0.1em\text{-}\cat{T}^c}
\nc{\MTc}{\Mod\kern-0.1em\text{-}\cat{T}^c}
\nc{\MT}{\Mod\kern-0.1em\text{-}\cat{T}}
\newcounter{enum-resume-hack}
\Crefname{Thm}{Theorem}{Theorems}
\Crefname{Prop}{Proposition}{Propositions}
\Crefname{Lem}{Lemma}{Lemmas}
\Crefname{thmx}{Theorem}{Theorems}

\begin{document}
\setlength{\parindent}{0cm}
\setlength{\parskip}{0.8ex}
\title[A symmetric monoidal fracture square]{A symmetric monoidal fracture square}

\author[Naumann]{Niko Naumann}
\author[Pol]{Luca Pol}
\author[Ramzi]{Maxime Ramzi}

\address{Niko Naumann, Fakult{\"a}t f{\"u}r Mathematik, Universit{\"a}t Regensburg, Universit{\"a}tsstraße 31, 93040 Regensburg, Germany}
\email{Niko.Naumann@mathematik.uni-regensburg.de}
\urladdr{https://homepages.uni-regensburg.de/$\sim$nan25776/}

\address{Luca Pol, Fakult{\"a}t f{\"u}r Mathematik, Universit{\"a}t Regensburg, Universit{\"a}tsstraße 31, 93040 Regensburg, Germany}
\email{luca.pol@mathematik.uni-regensburg.de}
\urladdr{https://sites.google.com/view/lucapol/}

\address{Maxime Ramzi,
FachBereich Mathematik und Informatik, Universit{\"a}t M{\"u}nster, Einsteinstraße 62 Germany}
\email{mramzi@uni-muenster.de}
\urladdr{https://sites.google.com/view/maxime-ramzi-en/home}

\begin{abstract}
Given a symmetric monoidal stable $\infty$-category $\CC$ which is rigidly-compactly generated and a set of compact objects $\K$ of $\CC$, one can form the subcategories of $\K$-complete and $\K$-local objects. The goal of this paper is to explain how to recover $\CC$ from its $\K$-local and $\K$-complete subcategories while retaining the symmetric monoidal
structure. Specializing to the case where $\CC$ is the $\infty$-category of $G$-spectra for a finite group $G$, our result can be viewed as a symmetric monoidal variant of the isotropy separation decomposition, a version of which appeared previously in \cite{Krause20}.
\end{abstract}

\subjclass[2020]{55P60, 55P91, 55U35}
\maketitle
\tableofcontents

\section{Introduction}

In both algebra and geometry, the concept of analyzing a category through an ``open-closed'' decomposition is foundational. This approach is formalized categorically by the notion of recollement, first introduced by Beilinson, Bernstein, and Deligne \cite{Deligne-Beilinson-Bernstein} in the context of derived categories of perverse sheaves, and later developed by Lurie in the language of $\infty$-categories \cite[A.8]{HA}. 

Recall that given a stable $\infty$-category $\cat X$ and stable subcategories $\cat U, \cat Z\subseteq \cat X$, we say that $(\cat U, \cat Z)$ is a recollement of $\cat X$ if the inclusion functors $i_* \colon \cat Z \hookrightarrow \cat X$ and $j_* \colon \cat U \hookrightarrow \cat X$ admit left adjoints $i^*$ and $j^*$ respectively such that $j^*i_*\simeq 0$, and the functors $j^*$ and $i^*$ are jointly conservative, see \cite[Definition A.8.1]{HA}.
We can represent the above situation by a diagram 
\[
\begin{tikzcd}
    \cat Z \arrow[r,"i_*"'] & \arrow[l, hook, bend right, "i^*"']\cat X \arrow[r, "j^*"] & \cat U \arrow[l,bend left, hook, "j_*"]
\end{tikzcd}
\]
where we refer to $\cat U$ as the \emph{open} part of $\cat X$, $\cat Z$ as the \emph{closed} part\footnote{For sheaves over of a topological space this corresponds to the usual open and closed decomposition but for quasi-coherent sheaves the direction is reversed.}, and to $\phi:=i^*j_*$ as the \emph{gluing} functor. One important feature of recollements is that they allow a description of the larger category $\cat X$ in terms of smaller, simpler categories $\cat Z$ and $\cat U$ with possibly complicated glueing functor $\phi$. More formally, one can show that there is a pullback square of stable $\infty$-categories 
\begin{equation}\label{pullback-intro}
\begin{tikzcd}
    \cat X \arrow[d] \arrow[r] & \Fun(\Delta^1, \cat Z)\arrow[d,"\mathrm{ev}_1"] \\
    \cat U \arrow[r,"\phi"] & \cat Z
\end{tikzcd}
\end{equation}
see \cite[Corollary 2.12]{shah2022recollementsstratification}.
Of particular interest are \emph{symmetric monoidal} recollements in which the localization functors $j_*j^*$ and $i_*i^*$ are compatible with the symmetric monoidal
structure, see for instance \cite[Definition 2.2]{shah2022recollementsstratification} or \cite{stratNC} for a more general form of this glueing procedure. In this case, the glueing functor is typically only \emph{lax} symmetric monoidal so the square (\ref{pullback-intro}) is not a pullback in symmetric monoidal stable $\infty$-categories.
Although this is sufficient for many purposes, such as studying algebraic structures within the larger category, it can pose challenges for other objectives, such as the analysis of coalgebraic or properadic structures. See for example \cite{recolldbl} for the study of dualizable objects in a symmetric monoidal recollement and \cite{Separablepaper} for the study of separable commutative algebras in the category of $G$-spectra.

In this paper we study symmetric monoidal recollements arising from a local duality context in the sense of \cite{BHV2018} and explain how one can express the category of interest as a pullback of symmetric monoidal stable $\infty$-categories. Before stating our main theorem we recall some notation. Let $\cat C\in \CAlg(\PrL)$ be a presentably symmetric monoidal stable $\infty$-category which is rigidly-compactly generated. Associated to a fixed set of compact objects $\K \subseteq \CC^\omega$, there are symmetric monoidal $\infty$-categories $\CC^{\K\text{-}\loc}$ and $\CC^{\K\text{-}\cpl}$ of $\K$-local and $\K$-complete objects in $\CC$ respectively, see \Cref{def-torsion-etc}, which fit into a symmetric monoidal recollement 
\[
\begin{tikzcd}
    \CC^{\K\text{-}\loc} \arrow[r,"i_{\loc}"'] & \arrow[l, bend right, "L"']\cat C \arrow[r, "\Lambda"] & \CC^{\K\text{-}\cpl}\arrow[l,bend left, "i_{\cpl}"].
\end{tikzcd}
\]
We observe that the glueing functor $\phi:=L i_{\cpl}$ is typically lax symmetric monoidal as the inclusion functor $i_{\cpl}$ is usually only lax. The key idea behind our approach is to replace the category of complete objects with 
\[
\widehat{\CC}_{\K}:= \Ind ((\CC^{\K\text{-}\cpl})^{\dual}),
\]
the Ind-completion of the subcategory of dualizable complete objects. 
It is not hard to see that the completion functor $\Lambda$ induces a symmetric monoidal left adjoint $F \colon \CC \to \widehat{\CC}_{\K} $. We can further localize the category $\widehat{\CC}_{\K}$ with respect to the set of compact objects $F(\K)$ and obtain a symmetric monoidal localization functor $\hat{L}$ which fits into a commutative diagram of the form:
\begin{equation}\label{mega-square-intro}
\begin{tikzcd}
    \CC \arrow[d,"F"'] \arrow[r,"L"] & \CC^{\K\text{-}\loc} \arrow[d,"F^{\loc}"] \\
    \widehat{\CC}_{\K}\arrow[r,"\hat{L}"] &  (\widehat{\CC}_\K)^{F(\K)\text{-}\loc}
\end{tikzcd}
\end{equation}
where $F^\loc$ is the functor induced by the universal property of $L$ by $F$. 

\begin{Thm*}[\ref{thm-pullback}]
    Let $\CC\in\CAlg(\PrL)$ be rigidly-compactly  generated and $\K\subset \CC^\omega$. Then the square (\ref{mega-square-intro}) is a pullback in $\CAlg(\PrL)$.
\end{Thm*}

We anticipate that the pullback decomposition described above will prove valuable across multiple areas of mathematics. As a first step in this direction, we highlight its utility in the context of equivariant homotopy theory.

For a finite group $G$, we let $\Sp_G$ denote the category of genuine $G$-spectra and consider $R \in \CAlg(\Sp_G)$ a commutative algebra object therein. Given a family $\cat F$ of subgroups of $G$, we would like to give an inductive description of the category 
\[
\Mod_{\Sp_G}(R)/\cat F
\]
of $R$-modules objects in genuine $G$-spectra with isotropy outside $\cat F$. To this end let $K\subseteq G$ be such that $K\notin\cat F$, but for all proper subgroups $K'\subsetneq K$ we have $K'\in\cat F$, and let $\cat F\cup K$ be the smallest family of subgroups of $G$ containing $\cat F$ and $K$. There is a canonical localization functor 
\[
L_{\cat F \cup K} \colon \Mod_{\Sp_G}(R)/\cat F \to \Mod_{\Sp_G}(R)/\cat F\cup K
\]
which nullifies the orbit corresponding to $K$. In order to retain the equivariant infomation at the subgroup $K$, we consider the $\mathbb E_\infty$-ring $\Phi^K(R)$ of $K$-geometric fixed points and recall that this is acted on by the Weyl group $W_G(K)$, allowing us to form the category
\[
\Perf(\Phi^K(R))^{hW_G(K)}.
\]
We will explain in \Cref{sec-G-spectra} how the geometric $K$-fixed points functor lifts to a functor:
\[
(\widetilde{\Phi}_R^K)^\omega \colon (\Mod_{\Sp_G}(R))^\omega \to \Perf(\Phi^K(R))^{hW_G(K)}.
\]
With this functor in hand we obtain the following result which can be interpreted as a symmetric monoidal variant of the isotropy separation decomposition. 

\begin{Thm*}[\ref{cor:isotropy_separation-compact}]
    In the above situation there is a pullback of symmetric monoidal stable $\infty$-categories:
\[
\begin{tikzcd}
    (\Mod_{\Sp_G}(R)/\cat F)^\omega \arrow[d,"(\widetilde{\Phi}_R^K)^\omega"'] \ar[r,"L^\omega_{\cat F\cup K}"] 
   &   (\Mod_{\Sp_G}(R)/\cat F\cup K )^\omega\arrow[d]\\
    \Perf(\Phi^K(R))^{hW_G(K)}\arrow[r]& \frac{\Perf(\Phi^K(R))^{hW_G(K)}}{\thickt{W_G(K)\circledast \Phi^K(R)}}
\end{tikzcd}
\]
where on the bottom right corner we used the induction functor from \Cref{not-induced}.
\end{Thm*}
The above pullback square provides a new way of performing isotropy separation, i.e. of doing induction on the family $\cat F$. For example, Krause \cite{Krause20} is able to study invertible objects in genuine $G$-spectra using this method, and in a companion paper \cite{Separablepaper} we exploit this pullback description to classify all separable  commutative algebras for the categories of $G$-spectra and derived $G$-Mackey functor for a finite $p$-group $G$. We observe that a version of \Cref{cor:isotropy_separation-compact} has already appeared in the literature in \cite[Theorem 3.11]{Krause20} under the restrictive assumption that the commutative algebra $R$ is inflated from a connective ring spectrum.

\subsection*{Outline}
In \Cref{section:DK} we collect some known results on Dwyer-Kan and Bousfield localizations. In \Cref{section:local}, we review in detail and generalize the setting of local duality from \cite{BHV2018}, and set the stage for the main result. There we use the contents of \Cref{section:atomic} which allow us to replace compact generation to a more general, relative notion of atomic generation. In \Cref{section:fracture} we prove the main result, \Cref{thm-pullback}, by an explicit analysis, using the preliminary sections.
In \Cref{section:EQ1}, we review enough equivariant stable homotopy theory to be able to apply our main result to this setting in \Cref{sec-G-spectra}. 
\subsection*{Acknowledgement}
The authors thank Marc Hoyois for useful conversations.
The first and second authors are supported by the SFB 1085 Higher Invariants in Regensburg. The third author was supported by the Danish National Research Foundation through the Copenhagen Centre for Geometry and Topology (DNRF151), and by the Deutsche Forschungsgemeinschaft (DFG, German Research Foundation) - Project-ID 427320536 - SFB 1442, as well as under Germany’s Excellence Strategy EXC 2044 390685587, Mathematics Münster: Dynamics-Geometry-Structure. 
\subsection*{Notations}
\begin{enumerate}
    \item Given a presentable $\infty$-category $\CC$ we write $\CC^\omega$ for its subcategory of compact objects. If $\CC$ admits a compatible symmetric monoidal structure, then we write $\CC^{\dual}$ for its subcategory of dualizable objects. 
    \item For a finite group $G$, we let $\Sp_G$ denote the $\infty$-category of genuine $G$-spectra. 
    \item We denote by $\CAlg$ the $\infty$-category of $\mathbb{E}_\infty$-algebras in spectra. 
    \item We denote by $\CAlg(\Cat_\infty^{\mathrm{perf}})$ the $\infty$-category of $2$-rings and symmetric monoidal exact functors. Recall that a $2$-ring is an essentially small, idempotent complete, symmetric monoidal and stable $\infty$-category whose tensor product is exact in each variable. 
    \item We denote by $\CAlg(\PrL)$ the $\infty$-category of presentably symmetric monoidal stable $\infty$-categories and symmetric monoidal left adjoints. 
    \item Given $\CC \in \CAlg(\PrL)$, we write $\unit$ for the unit object, $\hom_{\CC}$ for the internal hom object and $\mathbb{D}$ for the functional dual of $\CC$. 
\end{enumerate}

\section{Dwyer-Kan Localizations}\label{section:DK}
In this section we record some results about Dwyer-Kan localizations which we will use in the next sections. 

\begin{Def}
    Consider a functor of $\infty$-categories $f\colon \cat C\to \cat D$. 
    \begin{itemize}
    \item[(a)] We call $f$ a {\em Dwyer-Kan localization} if there exists a collection of arrows $W\subseteq \cat C$ such that $f$ witnesses $\cat D$ as $\cat C[W^{-1}]$, the $\infty$-categorical localization of $\cat C$ at $W$ in the sense of \cite[Definition 1.3.4.1]{HA}. In particular for any $\infty$-category $\cat E$, precomposition with $f$ induces a fully faithful functor 
    \[
   f^*\colon \Fun(\cat D, \cat E) \xhookrightarrow{} \Fun(\cat C, \cat E)
    \]
    whose essential image is given by those functors $\cat C \to \cat E$ which carry each morphism in $W$ to an equivalence in $\cat E$. 
    \item[(b)]  We call $f$ a {\em Bousfield localization} if it has a fully faithful right adjoint. 
    \item[(c)] We call $f$ a {\em Bousfield colocalization} if it has a fully faithful left adjoint.
    \end{itemize}
\end{Def}

\begin{Rem}\label{rem-localization-op}
Consider $f \colon \cat C \to \cat D$ and $f\op \colon \cat C \op \to \cat D \op$. 
    \begin{itemize}
        \item[(a)] If $f$ is a Dwyer-Kan localization, then so is $f\op$.
        \item[(b)] If $f$ is a Bousfield localization, then $f\op$ is a Bousfield colocalization. 
    \end{itemize}
\end{Rem}

The following is standard and also appears in \cite[Example 1.3.4.3]{HA}.

\begin{Lem}\label{lem:BousDK}
    Every Bousfield localization is a Dwyer-Kan localization. 
\end{Lem}

For general interest, we include the following list of examples of Dwyer-Kan localizations that are not Bousfield localizations, indicating that they abound in practice.

\begin{Exa}
\leavevmode
\begin{enumerate}
    \item We claim that the base-change functor $\Perf(\Z) \to \Perf(\Z[1/p])$ is a Dwyer-Kan localization. To show this we apply \cite[Theorem I.3.3]{Nikolaus-Scholze} and instead verify that the base-change functor is the Verdier quotient of $\Perf(\Z)$ by the subcategory of $p$-torsion $\Z$-modules:
    \[
    \Perf(\Z)^{p\text{-}\tors}:=  \ker(\Z[1/p]\otimes_{\Z}-)=\thick\langle \Z/p \rangle.
    \]
    By construction we have a factorization
    \[
    \begin{tikzcd}
        \Perf(\Z) \arrow[r, two heads] \arrow[d,"q"'] & \Perf(\Z[1/p])\\
        \frac{\Perf(\Z)}{\Perf(\Z)^{p\text{-}\tors}} \arrow[ur, dotted, two heads] &
    \end{tikzcd}
    \]
    which is essentially surjective as the base-change functor is so. Given perfect $\Z$-modules $M$ and $N$ and using the formula in \cite[Theorem I.3.3(ii)]{Nikolaus-Scholze}, we see that 
    \begin{equation}\label{torsion}
    \Map(qM, qN) = \colim_{Q \in \Perf(\Z)^{p\text{-}\tors}_{/N}} \Map_{\Z}(M, \cof(Q \to N)). 
    \end{equation}
    We now observe that any map $Q \to N$ must factors uniquely through $Q \to \Sigma^{-1}\Z/p^\infty\otimes_{\Z}Q \to N$ by the universal property of the $p$-torsion functor, and since $Q$ is compact the map $Q \to \Sigma^{-1}\Z/p^\infty\otimes_{\Z} Q $ must factor through a finite stage $Q \to \Sigma^{-1}\Z/p^n \otimes_{\Z}Q$. Therefore we can simplify the right hand side of (\ref{torsion}) as
    \[
     \colim_n \Map_{\Z}(M, \cof(\Sigma^{-1}\Z/p^n \otimes N \to N)).
    \]
    Since $\Sigma^{-1}\Z/p^n=\mathrm{fib} (\Z \xrightarrow{p^n}\Z)$, the above colimit is
    \[
     \colim (\Map_{\Z}(M,N) \xrightarrow{p} \Map_{\Z}(M, N)\xrightarrow{p}\ldots )
    \]
    which finally identifies with
    \begin{align*}
    \Map_{\Z}(M, \Z[1/p] \otimes N)
    & = \Map_{\Z[1/p]}(\Z[1/p]\otimes M , \Z[1/p]\otimes N)
    \end{align*}
    showing that the functor is also fully faithful, and so proving that the dotted map in the above diagram is an equivalence. 
    Next we claim that this localization functor does not admit a right adjoint $G$, and hence it is not a Bousfield localization.
    Otherwise, $G(\mathbb Q)$ would be a perfect $\mathbb Z$-module such that 
\[ \pi_0\Map_{\Perf(\mathbb Z)}(\mathbb Z, G(\mathbb Q))\simeq \pi_0\Map_{\Perf(\mathbb Q)}(\mathbb Q, \mathbb Q)\simeq\mathbb Q,\]
but maps between perfect $\mathbb Z$-modules form a finitely generated abelian group.
\item The canonical functor from simplicial sets to the $\infty$-category of spaces is a Dwyer-Kan localization, inverting the weak equivalences, but not a Bousfield localization (for example because it does not preverve colimits, i.e. there are colimits in simplicial sets that are not {\em homotopy} colimits). This example works with essentially any model (1-)category in place of simplicial sets.
\end{enumerate}
\end{Exa}

In fact, we also have:
\begin{Lem}\label{lem:DKBous}
    Let $f\colon \cat C\to \cat D$ be a Dwyer-Kan localization. If $f$ admits a right adjoint $f^R$, then $f^R$ is fully faithful and so $f$ is a Bousfield localization. 
\end{Lem}

\begin{proof}
Up to passing to a larger universe we can assume that both $\cat C$ and $\cat D$ are small. For typographical reasons let us write $g$ for the right adjoint $f^R$. We then need to show that $g$ is fully faithful. Recall from \Cref{rem-localization-op} that $f\op$ is a Dwyer-Kan localization so the map induced on presheaf categories
\[ 
(f\op)^*\colon \Fun(\cat D\op,\cat S)\to \Fun(\cat C\op,\cat S)
\]
is fully faithful.
Then $f\dashv g$ clearly implies $g\op\dashv f\op$ which in turn implies $(f\op)^*\dashv (g\op)^*$ \footnote{To see this note that the unit (resp. counit) map of $g\op\dashv f\op$ induce a counit (resp. unit) map for the adjunction $(f\op)^*\dashv (g\op)^*$. The passage $f\mapsto f^*$ is contravariant for 1-morphisms but covariant for 2-morphisms.}.
But the left adjoint of $(g\op)^*$ is $(g\op)_!$, the left Kan extension along $g\op$, hence $(g\op)_!\simeq (f\op)^*$ is fully faithful. This implies that $g\op$ (and $g$ too) is fully faithful as by \cite[Proposition 5.2.6.3]{HTT} there is a commutative square
\[
\begin{tikzcd}
    \Fun(\cat D\op, \cat S) \arrow[r, hook, "(g_!)\op"] & \Fun(\cat C\op, \cat S)\\
    \cat D \arrow[u, hook, "y"] \arrow[r,"g\op"] & \cat C \arrow[u, hook, "y"'].
\end{tikzcd}
\]
where the vertical arrows are given by the Yoneda embeddings.
\end{proof}

\begin{Cor}\label{cor:fully-faithful-adjunction}
    Let $i\colon \CC \rightleftarrows \DD : \Gamma$ be an adjunction and suppose that $\Gamma$ admits a right adjoint $R$. If $i$ is fully faithful, then so is $R$.
\end{Cor}

\begin{proof}
    Since $i\colon \CC\to \DD$ is fully faithful, $\Gamma$ is a Bousfield colocalization by definition. Hence, $\Gamma\colon \DD\op\to \CC\op$ is a Bousfield localization and thus a Dwyer-Kan localization by \Cref{lem:BousDK}. Since the notion of Dwyer-Kan localization is invariant under passage to the opposite categories, also $\Gamma\colon \DD\to\CC$ is a Dwyer-Kan localization. Then by \Cref{lem:DKBous}, $R$ is fully faithful.
    \end{proof}

\begin{Lem}\label{lem-equivalence-BC}
    Consider a commutative square 
  \[\begin{tikzcd}
	{\cat A_0} & {\cat  A_1} \\
	{\cat  B_0} & {\cat  B_1}
	\arrow["p", from=1-1, to=1-2]
	\arrow["{f_0}"', from=1-1, to=2-1]
	\arrow["{f_1}", from=1-2, to=2-2]
	\arrow["q"', from=2-1, to=2-2]
\end{tikzcd}\]
where $p,q$ are Dwyer-Kan localizations and $f_0,f_1$ admit right adjoints $f_0^R$ and $f_1^R$ respectively. If there exists an equivalence $p f_0^R\simeq f_1^Rq$, then the square is vertically right adjointable, that is, the Beck-Chevalley map $pf_0^R\to f_1^R q$ is an equivalence. 
\end{Lem}

\begin{proof}
    The equivalence $pf_0^R\simeq f_1^Rq$ guarantees that $f_0^R$ sends $q$-equivalences to $p$-equivalences. Thus $f_0^R$ induces a functor $g_1\colon  \cat B_1\to \cat A_1$, and the unit/counit maps, and triangle identities of the $f_0\dashv f_0^R$ adjunction induce unit/counit maps and triangle identities witnessing $f_1\dashv g_1$. By uniqueness of adjoints, the following specific map is an equivalence : $$g_1\to f_1^R f_1g_1 \to f_1^R$$
    where the first map is induced by the unit of the adjunction $f_1\dashv f_1^R$, and the second map is induced by the counit $f_1g_1\to \id$. The latter is completely determined by the fact that the composite $qf_0f_0^R\simeq f_1pf_0^R \simeq f_1g_1q\to q$ is $q$ applied to the counit. 
    Thus, precomposing the map in question with $q$ gives that the following composite is an equivalence  $$pf_0^R\simeq g_1q\to f_1^Rf_1g_1q \to f_1^Rq$$
    By the analysis above, this composite is the Beck-Chevalley map, so the latter is an equivalence, as was to be shown. 
\end{proof}

\section{Atomic objects and internal left adjoints}\label{section:atomic}
Throughout this section, we fix some $\DD\in\CAlg(\Pr^L_{\mathrm{st}})$ and work in the $\infty$-category of $\DD$-modules in $\Pr^L_{\mathrm{st}}$. After some preliminaries on $\DD$-linear categories, we recall the definition of $\DD$-atomic objects and list some important examples. Some of this material has already appeared in a slightly different setting in \cite{Ben-Moshe-Schlank} and \cite{Ben-Moshe} as well as in \cite{MaximeDbl}.

\begin{Rem}
Let $\MM$ and $\NN$ be $\DD$-modules in $\Pr^L_{\mathrm{st}}$. Then $\MM$ and $\NN$ are naturally tensored over $\DD$ and so we can consider the $\infty$-category $\Fun_{\DD}(\MM, \NN)$ of $\DD$-linear functors from $\MM$ to $\NN$, see \cite[Definition 4.6.2.7]{HA}. There is a canonical forgetful functor $\theta \colon\Fun_{\DD}(\MM,\NN)\to \Fun(\MM,\NN)$ and we denote by $\Fun_{\DD}^{L}(\MM,\NN)$ the full subcategory of $\DD$-linear functors $f\colon \MM \to \NN$ such that $\theta(f) \in \Fun^L(\MM,\NN)$. By \cite[Theorem 4.8.4.1]{HA}, evaluation at $\mathbb 1\in\DD\simeq\Mod_{\mathbb 1}(\DD)$ is an equivalence 
\[
 \Fun_{\DD}^L(\DD, \MM)\xrightarrow{\sim} \MM.
\]
Its inverse sends $m\in\MM$ to the colimit-preserving $\DD$-linear functor $m \otimes -\colon\DD \to \MM$.
\end{Rem}

\begin{Rem}\label{rem-D-lin-functors}
Let $\MM$ be a $\DD$-module in $\Pr^L_{\mathrm{st}}$, and $m\in\MM$. Then the $\DD$-linear functor
\[ -\otimes m \, :\, \DD \to \MM \]
admits a right adjoint, which we denote by $\hom_{\MM}(m,-)$ and refer to as {\em the internal hom functor}.
Explicitly, we have
\[ \Map_{\DD}(d,\hom_{\MM}(m,m'))\simeq \Map_{\MM}(d\otimes m, m'),\]
naturally in $d\in\DD$ and $m,m'\in\MM$.
In particular, if $f:\MM\to\NN$ is $\DD$-linear, then the map
\[ \Map_{\MM}(d\otimes m,m')\to\Map_{\NN}(f(d\otimes m),f(m'))\simeq\Map_{\NN}(d\otimes f(m),f(m'))\]
corresponds to a natural map
\[ \hom_{\MM}(m,m')\to\hom_{\NN}(f(m),f(m')),\]
which is an equivalence if $f$ is fully faithful.\\
An easy manipulation of adjoints shows that
for every functor $f\colon \MM\to \NN$ in $\Mod_{\DD}(\Pr^L_{\mathrm{st}})$ with right adjoint $f^R$, we have
\begin{equation}\label{adjoints-internalhom}
\hom_{\NN}(f(m),n)\simeq \hom_{\MM}(m,f^R(n)),
\end{equation}
naturally in  $m\in\MM$ and $n\in\NN$.
\end{Rem}

\begin{Exa}\label{ex-internalhom}
     Consider $\MM=\DD\in\Mod_{\DD}(\Pr^L_{\mathrm{st}})$ with $\DD$-module structure given by its symmetric monoidal structure. In this case the internal hom functor $\hom_{\MM}$ defined above agrees with the usual internal hom functor coming from the closed symmetric monoidal structure of $\DD$. 
\end{Exa}

\begin{Def}\label{def-proj-formula}
     Let $\MM,\NN$ denote $\DD$-modules in $\Pr^L_{\mathrm{st}}$, and let $f \in \Fun_{\DD}^L(\MM, \NN)$ with right adjoint $f^R$.
    For any $n\in \NN$ and $d\in\DD$, the counit map $ff^R(n) \to n$ induces a map $f(d\otimes f^R(n))\simeq d \otimes ff^R(n) \to d\otimes n$ whose mate we denote by
    \begin{equation}\label{proj-formula}
        \mathrm{pr}_{d,n}\colon d \otimes f^R(n) \to f^R(d \otimes n).
    \end{equation}
    We say that $f^R$ satisfies the \emph{projection formula} if the map (\ref{proj-formula}) is an equivalence for all $n \in \NN$ and $d\in \DD$. 
\end{Def}

It will be useful to have a more conceptual description of the projection map. 

\begin{Rem}\label{rem-proj-formula}
    Suppose we are in the situation of \Cref{def-proj-formula}. Since $f$ is $\DD$-linear, there is a commutative diagram 
    \[
    \begin{tikzcd}
        \MM \arrow[r,"f"] \arrow[d, "d \otimes-"'] & \NN \arrow[d,"d \otimes-"] \\
        \MM \arrow[r,"f"] & \NN
    \end{tikzcd}
    \]
    and so a natural equivalence of functors $ d \otimes f(-)\simeq f(d\otimes -)$. By assumption, the functor $f$ admits a right adjoint $f^R$, and so there is a Beck-Chevalley transformation of the above square which is a natural transformation $d \otimes f^R (-) \to f^R(d\otimes -)$. Unravelling the definition, we see that this Beck-Chevalley transformation agrees with the projection map. 
\end{Rem}

\begin{Lem}\label{megalemma-proj-formula}
\leavevmode
    \begin{enumerate}
        \item For $n\in \NN$ and $d,d'\in\DD$, the composite
        \[
        d \otimes d' \otimes f^R(n) \xrightarrow{d \otimes\mathrm{pr}_{d',n}} d \otimes f^R(d'\otimes n) \xrightarrow{\mathrm{pr}_{d, d'\otimes n}} f^R(d \otimes d' \otimes n)
        \]
        is homotopic to $\mathrm{pr}_{d\otimes d', n}$.
        \item For $n \in \NN$ and $d\in \DD$ dualizable, the projection map $\mathrm{pr}_{d,n}$ is an equivalence.
    \end{enumerate}
\end{Lem}

\begin{proof}
    For part (a), consider the following diagram
    \[
    \begin{tikzcd}
        \MM \arrow[r,"d' \otimes -"] \arrow[d, "f"'] & \MM \arrow[d,"f"] \arrow[r,"d\otimes -"]&  \MM \arrow[d,"f"]\\
        \NN \arrow[r,"d'\otimes -"] & \NN \arrow[r,"d\otimes -"]&\NN
    \end{tikzcd}
    \]
    which commutes by $\DD$-linearity of $f$.
    By \Cref{rem-proj-formula}, it suffices to verify that the composition of the two Beck-Chevalley transformations agrees with the Beck-Chevalley transformation associated to the outer square. This is a well-known property of Beck-Chevalley transformations.

    For part (b), consider $d \in \DD$ dualizable so that we have (co)evaluation maps $\ev\colon  \mathbb{D}(d) \otimes d \to \unit$ and $\mathrm{coev}\colon \unit \to d \otimes \mathbb{D}(d)$. For any $n \in \NN$, we define a map $\alpha_{d,n}\colon f^R(d \otimes n) \to d \otimes f^R(n)$ as the dotted arrow in the following square 
    \[
    \begin{tikzcd}[column sep=1.5cm, row sep=1.5cm]
    f^R(d \otimes n)\simeq \unit \otimes f^R(d \otimes n) \arrow[d,"{\mathrm{coev} \otimes 1}"']\arrow[r,dotted, "\alpha_{d,n}"]  & d \otimes f^R (n) \\
    d \otimes \mathbb{D}(d) \otimes f^R(d\otimes n) \arrow[r,"{1 \otimes \mathrm{pr}_{\mathbb{D}(d), d \otimes n}}"] & d \otimes f^R(\mathbb{D}(d)\otimes d \otimes n)  \arrow[u,"{1 \otimes f^R(\mathrm{ev} \otimes 1)}"'].
    \end{tikzcd}
    \]
    We claim that this map is an inverse for the projection map.  To see this, consider the following diagram:
    \[
    \begin{tikzcd}[column sep=2cm, row sep =2cm]
     d \otimes f^R(n) \arrow[r,"\mathrm{coev}\otimes 1 \otimes 1"] \arrow[d,"\mathrm{pr}_{d,n}"'] & d \otimes \mathbb{D}(d) \otimes d \otimes f^R(n) \arrow[dr, "1\otimes \mathrm{pr}_{\mathbb{D}(d) \otimes d,n}"]\arrow[d,"1\otimes 1 \otimes \mathrm{pr}_{d,n}"] \arrow[r,"1\otimes \mathrm{ev}\otimes 1"] & d \otimes f^R(n)\\
     f^R(d \otimes n) \arrow[r,"\mathrm{coev}\otimes 1"] & d \otimes \mathbb{D}(d) \otimes f^R(d\otimes n) \arrow[r, "1 \otimes \mathrm{pr}_{\mathbb{D}(d), d \otimes n}"] & d \otimes f^R(\mathbb{D}(d) \otimes d \otimes n) \arrow[u,"1 \otimes f^R(\mathrm{ev} \otimes 1)"']
    \end{tikzcd}
    \]
    where the top horizontal composite is the identity by definition of a duality datum, the left most square commutes by naturality of the projection map, the bottom triangle commutes by part (a), and the top triangle commutes as it is derived from the following commutative square 
    \[
    \begin{tikzcd}[column sep=1.5cm, row sep=1.5cm]
        \mathbb{D}(d) \otimes d \otimes f^R(n)\arrow[d,"\mathrm{pr}_{\mathbb{D}(d) \otimes d, n}"'] \arrow[r,"\mathrm{ev} \otimes 1"] &\unit \otimes f^R(n) \arrow[d,"{\mathrm{pr}_{\unit,n}}","="'] \\
        f^R(\mathbb{D}(d) \otimes d \otimes n) \arrow[r,"f^R(\mathrm{ev}\otimes 1)"] & f^R(\unit \otimes n).
    \end{tikzcd}
    \]
    The commutativity of the above big diagram shows that our map is a right inverse to the projection map. 
    To check that it is also a left inverse we consider the following diagram 
    \[
    \begin{tikzcd}[column sep=2cm, row sep=2cm]
        f^R(d \otimes n) \arrow[r,"f^R(\mathrm{coev}\otimes 1)"]\arrow[d,"\mathrm{coev}\otimes 1"'] & f^R(d \otimes \mathbb{D}(d) \otimes d \otimes  n)\arrow[r,"f^R(1 \otimes \mathrm{ev}\otimes 1)"] & f^R(d \otimes n)\\
        d \otimes \mathbb{D}(d) \otimes f^R(d\otimes n)\arrow[ur,"{\mathrm{pr}_{d \otimes \mathbb{D}(d), d \otimes n}}"'] \arrow[r,"{1\otimes \mathrm{pr}_{\mathbb{D(d)}, d \otimes n}}"] & d \otimes f^R(\mathbb{D}(d) \otimes d \otimes n) \arrow[u,"{\mathrm{pr}_{d, \mathbb{D}(d) \otimes d \otimes n}}"'] \arrow[r,"1 \otimes f^R(\mathrm{ev}\otimes 1)"] & d \otimes f^R(n)\arrow[u,"\mathrm{pr}_{d,n}"']
    \end{tikzcd}
    \]
    where the top horizontal composite is the identity by definition of a duality datum. The above diagram is commutative by the same reasoning as in the previous case, showing that our map is also a left inverse to the projection map.
\end{proof}

In the situation of \Cref{def-proj-formula}, the right adjoint $f^R$ is generally not a map in 
$\Mod_{\DD}(\Pr^L_{\mathrm{st}})$, indeed it may fail both to be a left-adjoint and to admit a $\DD$-linear structure. We now recall from \cite[Definition 5.1]{Ben-Moshe} the class of functors which behave well in this respect.

\begin{Def}\label{def:internal-left}
    Let $\MM,\NN$ denote $\DD$-modules in $\Pr^L_{\mathrm{st}}$, and let $f \in \Fun_{\DD}^L(\MM, \NN)$. We say that $f$ is an {\em internal left adjoint} if its right adjoint $f^R$ preserves colimits and satisfies the projection formula.
\end{Def}

\begin{Rem}
    Let $f$ be an internal left adjoint. By the discussion in \cite[Example 7.3.2.8 and Remark 7.3.2.9]{HA} the right adjoint $f^R$ acquires the structure of a $\DD$-linear functor and the unit and counit map of the adjunction $(f,f^R)$ can be promoted to $\DD$-linear natural transformation.
\end{Rem}

We observe that in favourable situations, the condition of being an internal left adjoint simplifies.

\begin{Lem}\label{lem:check-internal-adjoint}
    If $\DD$ is generated under colimits by dualizable objects, then $f$ is an internal left adjoint if and only if $f^R$ preserves colimits. 
\end{Lem}

\begin{proof}
This follows from \Cref{megalemma-proj-formula}(b).
\end{proof}

We now recall from \cite[Definition 5.4]{Ben-Moshe} a key finiteness property for objects in $\DD$-linear categories:

\begin{Def}\label{def:atomic}
    An object $m$ in a $\DD$-module $\MM\in\Mod_{\DD}(\Pr^L_{\mathrm{st}})$ is called $\DD$-\emph{atomic} if the associated $\DD$-linear functor $-\otimes m \colon \DD\to\MM$ is an internal left adjoint. When $\DD$ is clear from the context we simply say that $m$ is {\em atomic}. Explicitly, $m$ is $\DD$-atomic if and only if:
    \begin{enumerate}
        \item  The functor  
        \[ \hom_{\MM}(m,-):\MM\to\DD\]
    commutes with colimits, and 
    \item the projection map 
    \[ 
    \mathrm{pr}_{d, m'}\colon d\otimes\hom_{\MM}(m,m')\to \hom_{\MM}(m,d\otimes m')\]
    is an equivalence for all $d\in\DD$ and $m,m'\in\MM$.
    \end{enumerate}
    We denote by $\MM^{\DD\text{-}\mathrm{at}}$ the full subcategory of $\MM$ spanned by the $\DD$-atomic objects.
\end{Def}

We collect some examples of the above notion in the following lemma which has some overlap with \cite[Section 5]{Ben-Moshe}. 

\begin{Lem}\label{lem:atomics}
Consider $\DD \in \CAlg(\PrL)$ and $\MM \in \Mod_{\DD}(\PrL)$.
\begin{enumerate}
\item If $\DD=\Sp$, then $\MM^{\Sp\text{-}\mathrm{at}}=\MM^\omega.$
\item If $\MM=\DD$ as in \Cref{ex-internalhom}, then $\DD^{\DD\text{-}\mathrm{at}}=\DD^{\mathrm{dual}}.$
\item Internal left adjoints preserve atomic objects.
\item In general we have 
\[ \DD^{\dual}\otimes\MM^{\DD\text{-}\mathrm{at}}\subseteq \MM^{\DD\text{-}\mathrm{at}}. \]
\item Fully faithful functors in $\Mod_{\DD}(\Pr^L_{\mathrm{st}})$ detect $\DD$-atomic objects.
\item Fix a regular cardinal $\kappa$. We have $\MM^{\DD\text{-}\mathrm{at}}\subseteq \MM^\kappa$ for every $\DD$-module $\MM$ if and only if $\mathbb 1\in\DD^\kappa$. In particular, the category $\MM^{\DD\text{-}\mathrm{at}}$ is essentially small.
\end{enumerate}
\end{Lem}

\begin{proof}
\leavevmode
\begin{enumerate}
\item First note that the given $\MM\in\Pr^L_{\mathrm{st}}$ admits a unique $\Sp$-module structure because $\Sp\in\Pr^L$ is an idempotent commutative algebra whose modules are exactly the stable categories. This easily implies that in this case, the internal hom is the usual delooping of mapping spaces, i.e. $\Omega^\infty\hom_{\MM}\simeq \Map_{\MM}$. So condition (a) in \Cref{def:atomic}, is exactly that $m\in\MM$ is compact (note that $\Omega^\infty$ preserves all colimits as it preserves filtered colimits and it is exact) while condition (b) is
vacuous because $\Sp$ is generated under colimits by the dualizable unit $\mathbb 1\in\Sp$. 
\item Recall from \Cref{ex-internalhom} that when considering $\DD$ as a $\DD$-module, the internal hom functor identifies with the usual enrichment of $\DD$ over itself. Condition (b) of \Cref{def:atomic} with $m'=\unit$ then agrees with the definition of dualizability of $m$, hence $\DD^{\DD\text{-}\mathrm{at}} \subseteq \DD^{\dual}$. For the converse, note that if $m\in \DD$ is dualizable, then $\hom_{\DD}(d,-)\simeq \mathbb{D}(d) \otimes -$ preserves all colimits giving condition (a). \Cref{megalemma-proj-formula}(a) shows that the collection of $m'$ for which the map in condition (b) of \Cref{def:atomic} is an equivalence is closed under tensoring with objects of $\DD$, and it contains $\unit$ by definition of dualizability of $m$. Therefore it must be an equivalence for all $m, d \in \DD$ showing the reverse containment. 
\item Let $f\colon \MM\to  \NN$ be an internal left adjoint with right adjoint $f^R$, and let $m\in\MM$ be $\DD$-atomic. To check
that $f(m)\in\NN$ is $\DD$-atomic, we use (\ref{adjoints-internalhom}) to see that
\[ 
\hom_{\NN}(f(m),-)\simeq\hom_{\MM}(m,f^R(-))
\]
commutes with colimits. Next for any $d\in\DD$ and $n\in\NN$, we see that the projection map $d \otimes \hom_{\NN}(f(m),n) \to \hom_{\NN}(f(m), d \otimes n)$ identifies under the above equivalence and the $\DD$-linearity of $f^R$ with the projection map 
$d \otimes \hom_{\MM}(m, f^R(n)) \to \hom_{\MM}(m, d \otimes f^R( n))$ which is an equivalence by atomicity of $m$.
\item Let $m\in\MM$ be $\DD$-atomic and $d\in\DD^{\dual}$ be given. We claim that the $\DD$-linear colimit-preserving functor $d \otimes - \colon \MM \to \MM$ is an internal left adjoint. Using this and (c), we deduce that $d \otimes m$ is $\DD$-atomic proving (d). Let us then prove the claim. Using the duality datum of $d$ one constructs unit and counit maps which verify that the right adjoint of $d \otimes - \colon \MM\to \MM$ is given by $\mathbb{D}(d)\otimes - \colon \MM \to \MM$, which preserves colimits. For $d'\in \DD$ and $m \in \MM$, the projection map takes the form $d' \otimes \mathbb{D}(d) \otimes m \to \mathbb{D}(d)\otimes d' \otimes m$ and so an equivalence.
\item We assume $f\colon \MM\to\NN$ to be fully faithful in $\Mod_{\DD}(\Pr^L_{\mathrm{st}})$, $m\in\MM$ and $f(m)\in\NN$ is $\DD$-atomic, and aim to prove that $m\in\MM$ is $\DD$-atomic.
Firstly, 
\[\hom_{\MM}(m,-)\simeq\hom_{\NN}(f(m),f(-))\]
commutes with colimits, see (\ref{adjoints-internalhom}).  Secondly, for $d\in\DD$ and $m'\in\MM$, we see that the projection map $d \otimes \hom_{\MM}(m,m') \to \hom_{\MM}(m, d \otimes m')$ identifies under the above equivalence and the $\DD$-linearity of $f$ with the projection map 
$d \otimes \hom_{\NN}(f(m), f(n)) \to \hom_{\NN}(f(m), d \otimes f( n))$ which is an equivalence by atomicity of $f(m)$.
\item For the forward implication we take $\MM=\DD$ and note that $\unit\in \DD^{\DD\text{-}\mathrm{at}}\subseteq \DD^{\kappa}$. For the backward implication, we observe that for $m\in\MM$,
\[ \Map_{\MM}(m,-)\simeq\Map_{\DD}(\mathbb 1,\hom_{\MM}(m,-)),\]
so if $m\in\MM^{\DD\text{-}\mathrm{at}}$ and $\mathbb 1\in\DD^\kappa$, then $m\in\MM^\kappa$. For the final claim, since $\DD$ is presentable there is some $\kappa_0$ such that $\mathbb 1\in\DD^{\kappa_0}$, and consequently $\MM^{\DD\text{-}\mathrm{at}}\subseteq\MM^{\kappa_0}$ is essentially small, since $\MM$ is presentable.
\end{enumerate}    
\end{proof}

\section{Local duality}\label{section:local}
As in the previous section we fix some $\DD\in\CAlg(\Pr^L_{\mathrm{st}})$ and work in the $\infty$-category of $\DD$-modules in $\Pr^L_{\mathrm{st}}$. The goal of this section is to extend the notion local duality context developed in \cite{BHV2018} to the world of $\DD$-modules. Given a $\DD$-module $\MM$ and a set of atomic objects $\K \subseteq \MM$, we define $\infty$-categories of $\K$-torsion, $\K$-local and $\K$-complete objects in $\MM$. Our definitions reduces to those of~\cite{BHV2018} when choosing $\MM=\DD$.
After listing basic properties that these categories satisfy, we focus our attention on the special case $\MM=\DD$ and collect some technical observations about torsion and complete objects which will be used in the next section. 

\begin{Def}\label{def-torsion-etc}
    Let $\MM$ be a $\DD$-module in $\Pr^L_{\mathrm{st}}$, and let $\K\subseteq \MM^{\DD\text{-}\mathrm{at}}$.
    \begin{enumerate}
        \item We let $\MM^{\K\text{-}\mathrm{tors}}$ denote the full subcategory of $\MM$ generated under colimits and tensors with $\DD$ by $\K$, and call its objects {\em $\K$-torsion}. 
        \item We let $\MM^{\K\text{-}\mathrm{loc}}$ denote the full subcategory of $\MM$ spanned by the objects right orthogonal to $\MM^{\K\text{-}\mathrm{tors}}$, and call its objects {\em $\K$-local}.
        \item We let $\MM^{\K\text{-}\mathrm{cpl}}$ denote the full subcategory of $\MM$ spanned by the objects right orthogonal to $\MM^{\K\text{-}\mathrm{loc}}$, and call its objects {\em $\K$-complete}. 
    \end{enumerate}
\end{Def}

There is the following succinct characterization of the local objects:

\begin{Lem}\label{lem-local-obj}
In the situation of \Cref{def-torsion-etc}, we have
\[ \MM^{\K\text{-}\mathrm{loc}}=\{ m\in\MM\,\mid\, \forall k\in \K: \hom_{\MM}(k,m)=0\}.\]
\end{Lem}

\begin{proof}
To see the inclusion from left to right, choose some $m\in\MM^{\K\text{-}\mathrm{loc}}$, $k\in \K$ and $d\in\DD$ and compute
\[ \Map_{\DD}(d,\hom_{\MM}(k,m)))\simeq \Map_{\MM}(d\otimes k,m)=0,\]
since $d\otimes k$ is torsion and $m$ is local. Since this holds for all $d\in\DD$, we conclude $\hom_{\MM}(k,m)=0$.\\
To see the reverse inclusion, fix some $m$ in the right hand side and consider
\[ \MM\supseteq\MM':=\{ m'\in\MM\,\mid\, \Map_{\MM}(m',m)=0\}.\]
To see that $m$ is local, we need to check that $\MM^{K\text{-}\mathrm{tors}}\subseteq \MM'.$ It is clear that $\MM'\subseteq\MM$ is closed under colimits, so it suffices to show that $\DD\otimes \K\subseteq\MM'$. To see this,
take some $d\in\DD, k\in \K$ and compute
\[ \Map_{\MM}(d\otimes k,m)\simeq \Map_{\DD}(d,\hom_{\MM}(k,m))=0.\]
\end{proof}

We record basic properties of the subcategories introduced in \Cref{def-torsion-etc}:

\begin{Prop}\label{lem:setuptors}
    Let $\MM$ be a $\DD$-module in $\PrL$, and let $\K\subseteq \MM^{\DD\text{-}\mathrm{at}}$ be a subset. 
    \begin{enumerate}
    \item The $\infty$-category $\MM^{\K\text{-}\mathrm{tors}}$ is presentable, and canonically a $\DD$-module in $\Pr^L_{\mathrm{st}}$ such that the inclusion $i_{tors}\colon \MM^{\K\text{-}\mathrm{tors}}\to \MM$ is $\DD$-linear and it acquires the structure of an internal left adjoint. We write $\Gamma=\Gamma_\K$ for the right adjoint of $i_{\mathrm{tors}}$.
    \item The category $\MM^{\K\text{-}\mathrm{loc}}$ is presentable and canonically a $\DD$-module in $\Pr^L_{\mathrm{st}}$ such that the inclusion $i_{\mathrm{loc}}\colon
    \MM^{\K\text{-}\mathrm{loc}}\to \MM$ is $\DD$-linear. The functor $i_{\loc}$ admits a right adjoint $\Delta=\Delta_{\K}$ and a left adjoint $L=L_{\K}$, and the latter acquires the structure of an internal left adjoint. 
    \item The subcategory $\MM^{\K\text{-}\mathrm{cpl}}\subseteq \MM$ is an accessible localization, so it is presentable and the inclusion $i_{\mathrm{cpl}}:\MM^{\K\text{-}\mathrm{cpl}}\to \MM$
    preserves limits. In particular $i_{\cpl}$ admits a left adjoint $\Lambda=\Lambda_\K$. Furthermore, the subcategory $\MM^{\K\text{-}\mathrm{cpl}}$ admits a $\DD$-module structure such that $\Lambda$ can be promoted to a $\DD$-linear functor.
    \item The units and counits of these adjunction participate in a fibre sequences of functors $  i_{\tors}\Gamma \to \mathrm{id}_{\MM}\to i_{\loc}L$ and $i_{\loc} \Delta \to \mathrm{id}_{\MM} \to i_{\cpl} \Lambda$.
    \item There is a diagram of adjoints
    \[
    \qquad\qquad
    \begin{tikzcd}
        \MM^{\K \text{-}\loc} \arrow[r, shift right, "i_\loc"'] & \arrow[l, shift right, "L"']\MM \arrow[r, shift right, "\Gamma"'] & \MM^{\K\text{-}\tors}\arrow[l, shift right, "i_\tors"'] 
    \end{tikzcd}
    \quad \text{in} \;\; \Mod_{\DD}(\PrL)
    \]
     where the top composite is a cofibre sequence, and the bottom composite is a fibre sequence. 
    \item There is a diagram of adjoints
     \[
    \begin{tikzcd}
        \MM^{\K \text{-}\loc} \arrow[r, shift left, "i_\loc"] & \arrow[l, shift left, "\Delta"]\MM \arrow[r, shift left, "\Lambda"] & \MM^{\K\text{-}\cpl}\arrow[l, shift left, "i_\cpl"] 
    \end{tikzcd} 
    \]
    where the top composite is a cofibre sequence in $\Mod_{\DD}(\PrL)$, and the bottom composite is a fibre sequence in $\PrL$.
    \end{enumerate}
\end{Prop}

\begin{proof}
\leavevmode
\begin{enumerate}
    \item Since $\DD$ is presentable, there is a cardinal $\kappa$ such that $\DD$ is $\kappa$-compactly generated. Then $\MM^{\K\text{-}\mathrm{tors}}$ is generated in $\MM$ under colimits by a set of objects, namely $\DD^\kappa\otimes \K$, and is therefore presentable.\footnote{In general, if $\CC$ is presentable and $\CC_0$ is an essentially small subcategory, then the subcategory $\widetilde\CC_0\subseteq \CC$ generated under colimits in $\CC$ by $\CC_0$ is presentable. To see this, one chooses a regular cardinal $\kappa$ such that both $\CC$ is $\kappa$-compactly generated and all objects in $\CC_0$ are $\kappa$-compact. One then checks that $\widetilde \CC_0=\Ind_\kappa(\overline{\CC_0})$, where $\overline{\CC_0}$ is the subcategory of $\CC$ generated under $\kappa$-small colimits from $\CC_0$.}
It is further closed in $\MM$ under $\DD$-tensors and is thus uniquely a $\DD$-module in a way which makes $i_{\mathrm{tors}}$ a $\DD$-linear functor. It is clear by construction that $i_{\mathrm{tors}}$ preserves (and creates) colimits and so it admits a right adjoint $\Gamma$. 
    To prove that the inclusion is an internal left adjoint, it remains to see that the right adjoint $\Gamma$ preserves colimits and satisfies the projection formula. We note that each $k\in \K$ is also atomic in $\MM^{\K\text{-}\mathrm{tors}}$, by assumption on $\K$ and \Cref{lem:atomics}(e), and that the functors $\{ \hom_{\MM^{\K\text{-}\mathrm{tors}}}(k,-)\}_{k\in \K}$ are jointly conservative on $\MM^{\K\text{-}\mathrm{tors}}$ as torsion objects are generated under colimits and $\DD$-tensors by the atomic object of $\K$.
    So we can check both claims by computing suitable internal homs out of arbitrarily given $k\in \K$. For the commutation of $\Gamma$ with colimits, we have
    \begin{align*}
    \hom_{\MM^{\K\text{-}\mathrm{tors}}}(k,\colim_j \Gamma f(j)) & \simeq \colim_j\hom_{\MM^{\K\text{-}\mathrm{tors}}}(k,\Gamma(f(j)))\\
     & \simeq \colim_j\hom_{\MM}(i_{\mathrm{tors}}(k), f(j))\\
     & \simeq\hom_{\MM}(i_{\mathrm{tors}}(k), \colim_j f(j))\\
     & \simeq  \hom_{\MM^{\K\text{-}\mathrm{tors}}}(k, \Gamma(\colim_j f(j))).\\
     \end{align*} 
This shows that $\Gamma$ preserves colimits, and a similar computation, which we leave to the reader, establishes the projection formula.
\item We can describe $\MM^{\K\text{-}\mathrm{loc}}\subseteq \MM$ as the $W$-local objects in $\MM$, where $W$ is the set of maps $\{ d\otimes k\to 0, d\in\DD^\kappa, k\in \K$\} for $\kappa$ chosen as in part (a). It follows that $\MM^{\K\text{-}\mathrm{loc}}$ is an accessible localization of $\MM$, and so presentable, and the inclusion $i_{\mathrm{loc}}$ has a left adjoint $L$. Using the definition of local objects given in \Cref{lem-local-obj} and atomicity of $\K$, one sees that $\MM^{\mathrm{loc}}\subseteq\MM$ is stable under colimits, and hence $i_{\mathrm{loc}}$ preserves (and creates) colimits. In particular $i_\loc$ admits a right adjoint $\Delta$. It remains to prove that local objects form a $\DD$-submodule of $\MM$ and that $L$ is an internal left adjoint, which we postpone for the moment. 
\item The first part uses the same argument as in (b) using that $\MM^{\K\text{-}\mathrm{loc}}$ is presentable and thus generated under colimits by a small set of objects. In particular we deduce the existence of the left adjoint $\Lambda$. As for local objects, we postpone the proof that $\MM^{\K\text{-}\cpl}$ is a $\DD$-module and that $\Lambda$ is $\DD$-linear. 
\item We next prove that for every $m\in\MM$, the diagram determined by the (co)unit maps
\[ 
i_{\tors }\Gamma(m)\to m \to i_{\loc}L(m)
\]
is a fiber sequence. As there are no nonzero map between torsion and local objects, we obtain a map
\[ 
f\colon i_{\tors }\Gamma(m)\to F(m):=\mathrm{fib}(m\to i_{\loc}L(m)).
\]
Note that $i_{\tors}\Gamma i_\loc L(m)\simeq 0$ as it is both torsion and local. It follows that $F(m)$ is torsion, hence the universal property of $i_\tors\Gamma(m)\to m$ yields a map $F(m)\to i_\tors\Gamma(m)$, which is inverse to $f$.
 The argument for the existence of a fibre sequence $i_{\loc}\Delta \to \mathrm{id}_{\MM} \to i_{\cpl}\Lambda$ is similar: since $\Map_{\MM}(i_{\loc}\Delta (m), i_{\cpl}\Lambda (m))=0$, we get map 
\[
g \colon C(m):=\cof (i_{\loc}\Delta(m) \to m) \to i_{\cpl}\Lambda(m).
\]
A simple yoga with adjunctions shows that $C(m)$ is complete, so we get an induced map $i_{\cpl}\Lambda(m) \to C(m)$ by the universal property of $\Lambda$ which is an inverse for $g$.   
\item We first verify that the diagram in (e) gives (co)fibre sequences in $\PrL$. Now recall that limits in $\PrL$ are calculated in $\Cat_\infty$ and that colimits are calculated by passing to right adjoints and taking the limit of the corresponding diagram. Thus we are only left to check that $\MM^{\K\text{-}\loc} \to \MM \to \MM^{\K\text{-}\tors}$ is a fibre sequence, which is clear by the existence of the fibre sequence $i_{\tors}\Gamma m \to m \to i_{\loc}Lm$ for all $m \in \MM$. 
In particular this means that $\MM^{\K\text{-}\loc}$ is a Verdier quotient of the $\DD$-module $\MM$ by the $\DD$-submodule $\MM^{\K\text{-}\tors}$. Since $\Mod_{\DD}(\PrL)$ is closed under colimits and these are preserved by the forgetful functor $\Mod_{\DD}(\PrL) \to \PrL$, we deduce that $\MM^{\K\text{-}\loc}$ admits a $\DD$-module structure making the functor $L$ into a $\DD$-linear functor. We next show that $L$ is an internal left adjoint which in turn implies $\DD$-linearity of $i_{\loc}$. We have already check in (b) that $i_\loc$ preserves colimits so we are only left to check that the projection formula holds. Unravelling the definitions, we need to verify that the projection map $d \otimes i_{\loc}(l) \to i_{\loc}(d \otimes l):=i_{\loc}(L(d \otimes i_{\loc}(l)))$ is an equivalence for all $d \in \DD$ and local objects $l$. Using the definition of local object given in \Cref{lem-local-obj} and atomicity of $\K$ one sees that the source of the projection map is local (as the target). So we can check that the map is an equivalence after applying $L$. This is now clear since the map simplifies to the identity map. We have finally proved all of part (b). In particular the diagram in (e) lift to $\Mod_{\DD}(\PrL)$ as claimed and it is a (co)fibre sequence in $\DD$-modules by the above discussion and the behaviour of (co)limits in module categories. This concludes part (e).
\item  Next, consider the diagram of adjoints in (f). By the existence of the fibre sequence $i_{\loc}\Delta \to \mathrm{id}_{\MM} \to i_{\cpl}\Lambda$, it is clear that the bottom composite is a fibre sequence in $\PrL$, and so the top composite is a cofibre sequence in $\PrL$ too. In particular this means that $\MM^{\K\text{-}\cpl}\subseteq\MM$ is a Verdier quotient of the $\DD$-module $\MM$ by the $\DD$-submodule $\MM^{\K\text{-}\mathrm{loc}}$ and so $\MM^{\K\text{-}\cpl}$ admits a $\DD$-module structure making the functor $\Lambda$ into a $\DD$-linear functor proving also the last claim of (c). This also concludes part (f).
\end{enumerate} 
\end{proof}

\begin{Rem}
    We warn the reader that the inclusion functor is not $\DD$-linear in general.  For an example take $\MM=\DD=\mathsf{D}(\Z)$ and $\K=\{\Z/p\}$ so that $\MM^{\K\text{-}\cpl}\simeq \mathsf{D}(\Z)^{p\text{-}\cpl}$ is the $\infty$-category of derived $p$-complete $\Z$-modules. Unravelling the definitions, the $\DD$-linearity of $i_\cpl$ would give us an equivalence
    \[
    m \otimes_{\Z} c \simeq \Lambda(\Lambda(m) \otimes_{\Z} c))
    \]
    for all $m\in \mathsf{D}(\Z)$ and $c\in \mathsf{D}(\Z)^{p\text{-}\cpl}$. Taking $m=\bbQ$ and $c=\unit=\Z_p$ in the above, we would get an equivalence 
    \[
    \mathbb{Q} \otimes_{\Z} \Z_p \simeq \Lambda(\Lambda(\mathbb{Q}) \otimes_{\Z} \Z_p)\simeq 0
    \]
    which is a contradiction.
\end{Rem}

We record the functoriality of the above localizations.

\begin{Lem}\label{lem:functoriality-of-loc}
Let $\MM$ and $\MM'$ be $\DD$-modules in $\Pr^L_{\mathrm{st}}$, let $\K \subseteq\MM^{\DD\text{-}\mathrm{at}}$ be a subset of atomic objects and $f\colon \MM\to\MM'$ an internal left adjoint. 
\begin{itemize}
\item[(a)] We have $f(\K)\subseteq(\MM')^{\DD\text{-}\mathrm{at}}$ and $f(\MM^{\K\text{-}\mathrm{tors}})\subseteq(\MM')^{f(\K)\text{-}\mathrm{tors}}$
\item[(b)] The composition $L'\circ f$ factors uniquely through $L$, and gives a commutative diagram of internal left adjoints
\[
\begin{tikzcd}
    \MM \ar[r, "L"]\ar[d, "f"']& \MM^{\K\text{-}\mathrm{loc}} \ar[d, "f^{\mathrm{loc}}"] \\
    \MM'\ar[r, "L'"] & (\MM')^{f(\K)\text{-}\mathrm{loc}}.
    \end{tikzcd}
\]
\item[(c)] There is a map of fiber sequences of functors 
\[ \begin{tikzcd}
    f i_\tors \Gamma \ar[r]\ar[d,"\alpha"', dashed] & f \ar[r]\ar[d, "="] & f i_\loc L\ar[d,"\beta", dashed]\\
    i_\tors'\Gamma' f \ar[r] & f \ar[r] & i'_\loc L'f, 
    \end{tikzcd}
    \]
    and the following are equivalent
    \begin{itemize}
         \item[i)] $\alpha$ is an equivalence.
        \item[ii)] $\beta$ is an equivalence.
        \item[iii)] $f$ preserves local objects.
        \end{itemize}
\item[(d)] Assume that $f\colon \DD\to\DD'$ is a map in $\CAlg(\Pr^L_{\mathrm{st}})$ whose right adjoint preserves colimits and satisfies the projection formula. Note this is a special case of the above by taking $\MM:=\DD$ and $\MM':=\DD'$ endowed with the evident $\DD$-linear structures. Then the equivalent conditions in part (c) hold true. Moreover, $L$ is smashing, i.e. 
\[ i_\loc L \simeq (i_\loc L)(\mathbb 1_{\DD})\otimes -.\]
Finally, the Beck-Chevalley transformations
\[ 
Lf^R\xrightarrow{\sim}(f^\loc)^RL' \quad \mathrm{and} \quad  f i_\loc\xrightarrow{\sim} i'_\loc f^\loc
\]
are equivalences.
\end{itemize}
\end{Lem}

\begin{proof}
\leavevmode
    \begin{itemize}
        \item[(a)] The first inclusion is a restatement of \Cref{lem:atomics}(d). For the second, we observe that 
        \[ 
        X:=\{m\in\MM\,\mid\, f(m)\in(\MM')^{F(\K)\text{-}\mathrm{tors}}\} \subseteq \MM
        \]
        is stable under colimits and tensors with $\DD$. 
        Since clearly $\K\subseteq X$, we get $\MM^{\K\text{-}\mathrm{tors}}\subseteq X$, i.e. $f(\MM^{\K\text{-}\mathrm{tors}})\subseteq (\MM')^{f(\K)\text{-}\mathrm{tors}}$.
        \item[(b)] By assumption and \Cref{lem:setuptors} (b), the composition $L'\circ f$ is an internal left adjoint, and it obviously annihilates $\K$. This formally implies the claim.
        \item[(c)] We have fiber sequences $i_\tors\Gamma\to\id_{\MM}\to i_\loc L$ 
        and $i'_\tors \Gamma'\to\id_{\MM'}\to i'_\loc L'$ which 
        by post and pre-composing with $f$ yield the displayed fiber sequences. Part (a)
        implies that the composition $f i_\tors \Gamma\to i'_\loc L'f$ is zero, giving rise to $\alpha$ and $\beta$ along with an identification 
        $\mathrm{fib}(\alpha)\simeq\Sigma^{-1}\mathrm{fib}(\beta)$ of functors valued in torsion objects. This makes 
        the equivalence of $i)$ and $ii)$ evident. If $iii)$ holds, then $\mathrm{fib}(\beta)$ is additionally local, hence 
        zero, showing that $i)$ and $ii)$ hold true. If $i)$ and $ii)$ hold true, and $m\in\MM$ is local, 
        then we see that
        \[ f(m)\simeq f(i_\loc Lm)\stackrel{\beta}{\simeq} i'_\loc L'f(m)
        \] 
        is local, showing that   $(iii)$ holds.
        \item[(d)] In this special case, the projection formula shows that 
        \[ f^Rf\simeq f^R(\mathbb 1_{\DD'})\otimes (-)\]
        holds, which implies that $f^Rf$ preserves local objects because local objects are a $\DD$-submodule. Now, if $d\in\DD$ is local and $k\in{\K}$ is arbitrary, this implies that
        \[ \hom_{\DD'}(f(k),f(d))\simeq\hom_{\DD}(k,f^Rf(d)))=0,\]
        showing that $f(d)$ is local. Thus $f$ preserves local objects. Using that the composite functor $i_\loc L$ is $\DD$-linear, we see that
        \[ i_\loc L(-)\simeq  i_\loc L(\unit_{\DD} \otimes -)\simeq (i_\loc L)(\mathbb 1_{\DD})\otimes (-)\]
        showing that $L$ is smashing.
        The same argument also applies to $\DD'$ showing that $L'$ is also smashing.  
        For the first BC-transformation, we compute
        \begin{align*} 
        i_\loc L f^R  & \simeq  (i_\loc L)(\mathbb 1_{\DD})\otimes f^R(-)  & &\text{since } L \text{ is smashing}\\
                      & \simeq  f^R ( (fi_\loc L)(\mathbb 1_{\DD})\otimes -) & &\text{since }f^R \text{ is } \DD\text{-linear} \\
                      & \simeq  f^R((i'_\loc L'f)(\mathbb 1_{\DD})\otimes -)  & & \text{by part (c)}\\
                      & \simeq  f^R((i'_\loc L')( \mathbb 1_{\DD'})\otimes -) & &\text{as } 
 f \text{ is symmmetric monoidal} \\
                      & \simeq  (L'f)^RL' & & \text{since } L' \text{  is smashing}\\
                      & \simeq  (f^\loc L)^RL' & &  \text{by part (b)}\\
                      & \simeq  i_\loc(f^\loc )^RL'. & & \\
        \end{align*}
        Since $i_\loc$ is fully faithful we have an abstract equivalence $L f^R \simeq (f^{\loc})^R L'$, and so we can conclude by \Cref{lem-equivalence-BC}. For the final claim, we note that the require BC-map is given by the composite 
\[
fi_\loc x \xrightarrow[\sim]{\mathrm{unit} }i'_\loc L' f i_\loc x\xrightarrow{\sim} i'_\loc f^\loc Li_\loc x \xrightarrow[\sim]{ \mathrm{counit}} i'_\loc f^\loc x 
\]
where the middle arrow is the natural equivalence witnessing the commutativity of the above diagram. We observe that the first map is an equivalence since $f$ preserves local objects. The last map is an equivalence since $L$ is a localization.
        \end{itemize}
\end{proof}

We next establish the following generalization of \cite{DwyerGreenlees02}.

\begin{Lem}\label{lem:cpltors}
    Let $\MM$ be a $\DD$-module in $\Pr^L_{\mathrm{st}}$, and let $\K$ be a set of atomic objects of $\MM$. Then the functors $\Lambda i_{\mathrm{tors}}\colon\MM^{\K\text{-}\mathrm{tors}}\to \MM^{\K\text{-}\mathrm{cpl}}$ and $\Gamma i_{\mathrm{cpl}}\colon\MM^{\K\text{-}\mathrm{cpl}}\to \MM^{\K\text{-}\mathrm{tors}}$ are mutually inverse equivalences. 
\end{Lem}

\begin{proof}
We have a chain of adjunctions $i_{\mathrm{tors}}\dashv \Gamma\dashv R$, as $\Gamma$ is colimit-preserving by \Cref{lem:setuptors}(a). \Cref{cor:fully-faithful-adjunction} shows that $R$ is fully faithful, so that $\Gamma$ is a  Bousfield localization.
Since all the categories considered are stable, it suffices to prove that its kernel consists of $\MM^{\K\text{-}\mathrm{loc}}=\mathrm{ker}(\Lambda)$. Now $\Gamma(m)=0$ if and only if for all $x\in\MM^{\K\text{-}\mathrm{tors}}$, $\hom_{\MM^{\K\text{-}\mathrm{tors}}}(x,\Gamma(m)) = 0$ if and only if this is so for all $x\in \K$, if and only if $\hom_{\MM}(i_{\mathrm{tors}}(x),m)=0$, if and only if $m\in \MM^{\K\text{-}\mathrm{loc}}$. 
    \end{proof}

We next recall that for a finite set of atomic objects, we have the following Morita-style description of the torsion objects.

\begin{Lem}
Let $\MM\in\Mod_{\DD}(\Pr^L_{\mathrm{st}})$ be given, and let $m\in \MM^{\DD\text{-}\mathrm{at}}$ be an atomic object of $\MM$. Then there is a $\DD$-linear equivalence
\[ \mathrm{RMod}_{\mathrm{end}_{\MM}(m)}(\DD) \xrightarrow{\sim} \MM^{\{m\}\text{-}\mathrm{tors}}\]
sending $\mathrm{end}_{\MM}(m)$ to $m$.
\end{Lem}

\begin{proof}
This is a direct application of \cite[Proposition 4.8.5.8]{HA}, 
taking $\MM=\MM^{\{m\}\text{-}\mathrm{tors}}$ and $\CC=\DD$ there.
Most assumptions of that result are clearly met, while points (3),(4) and (6) are 
fulfilled by the definition of $m$ being atomic. It remains
to check (5) which requires that $\hom_{\MM}(m,-)$ be conservative on $\MM^{\K\text{-}\mathrm{tors}}$. Indeed, if a torsion object $t$ satisfies
$\hom_{\MM}(m,t)=0$, then it is also 
local, and hence $t=0$. 
\end{proof}

\subsection{A special case}

In this subsection we specialize some definitions and results from the previous subsection. We will work with $\CC \in \CAlg(\PrL)$, and write $\unit$ for the unit object, $\hom_{\CC}$ for the internal hom object and $\mathbb{D}$ for the functional dual.

\begin{Ter}\label{initialterm}
        We say that $\cat C\in\CAlg(\Pr^L_{\mathrm{st}})$ is \emph{compactly generated} if there exists a small $\infty$-category $\CC_0$ which admits finite colimits and an equivalence $\CC\simeq \Ind(\CC_0)$. In this case, we can in fact choose $\CC_0$ to be the $\infty$-category $\CC^\omega$ of compact objects of $\CC$. If $\CC_0$ can additionally be chosen to consists of dualizable objects, then we say that $\CC$ is \emph{compactly generated by dualizable objects}. In this case it follows that $\CC^\omega\subseteq \CC^{\dual}$\footnote{To see this, write a compact object $x$ as a filtered colimit of dualizables and use compactness to see that $x$ is a retract of one of these dualizable objects, hence itself dualizable.} . We also say that $\CC$ is \emph{rigidly-compactly generated} if it is compactly generated and $\CC^\omega=\CC^{\dual}$. 
\end{Ter}

\begin{Rem}
    Being rigidly-compactly generated is also equivalent to being compactly generated by dualizable objects and $\unit\in \CC^\omega$. The forward implication is clear and for the backward implication it is enough to show that any dualizable object $x$ is compact. The fact that $\unit$ is compact together with $\Map_{\CC}(x,-))\simeq \Map_{\CC}(\unit, \hom_{\CC}(x,-))$, let us reduce to proving that $\hom_{\CC}(x,)\simeq \mathbb{D}x \otimes -$ preserves filtered colimits, which is clear. 
\end{Rem}

The next definition is the special case of \Cref{def-torsion-etc} in case $\DD=\MM=\CC$ there, but we paraphrase it for convenience.

\begin{Def}
    Let $\cat C\in\CAlg(\Pr^L_{\mathrm{st}})$ be compactly generated by dualizable objects, and let $\K\subseteq \cat C$ be a set of compact objects of $\cat C$.     \begin{itemize}
        \item We say that $x\in\cat C$ is $\K$-\emph{torsion} if $x$ belongs to the localizing tensor ideal generated by $\K$ in $\cat C$. The full subcategory of $\K$-torsion object is denoted by $\cat C^{\K\text{-tors}}\subseteq \cat C$. 
        \item We say that $x\in \cat C$ is $\K$-\emph{local} if lies in the right orthogonal of $\cat C^{\K\text{-tors}}$, equivalently when $\hom_{\CC}(k, x)\simeq 0$ for all $k\in\K$. The full subcategory of $\K$-local objects is denoted by $\cat C^{\K\text{-loc}}\subseteq\CC$. 
        \item We say that $x$ is $\K$-\emph{complete} if it lies in the right orthogonal of $\cat C^{\K\text{-loc}}$. The full subcategory of $\K$-complete objects is denoted by $\cat C^{\K\text{-cpl}}\subseteq\CC$. 
    \end{itemize}
    We let $i_{\text{tors}}$, $i_{\text{loc}}$ and $i_{\text{cpl}}$ denote the canonical inclusion functors into $\cat C$. If $\K$ is clear from context, it will be omitted from the notation. For instance we will simply write $\CC^{\tors}, \; \CC^{\cpl}$ and $\CC^{\loc}$ for the category of torsion, complete, and local objects, respectively.
\end{Def}

\begin{Lem}\label{lem-properties-torsion-local}
     Let $\cat C\in\CAlg(\Pr^L_{\mathrm{st}})$ be compactly generated by dualizable objects and let $\K\subseteq \cat C$ be a set of compact objects of $\cat C$.  
     \begin{itemize}
     \item[(a)] The functor $i_{\text{tors}}$ has a right adjoint $\Gamma=\Gamma_{\K}$, the functor $i_{\text{loc}}$ has a left adjoints $L=L_{\K}$ and a right adjoint $\Delta=\Delta_\K$, and the functor $i_{\text{cpl}}$ has a left adjoint $\Lambda=\Lambda_{\K}$. Furthermore, the units and counits of these adjunction participate in natural fibre sequences 
     \[
     \Gamma \to \mathrm{id}_{\CC} \to L \quad \mathrm{and} \quad \Delta \to \mathrm{id}_{\CC} \to \Lambda.
     \]
     Finally, there is a diagram of adjoints
   \[
    \qquad\qquad
    \begin{tikzcd}
        \CC^{\K \text{-}\loc} \arrow[r, shift right, "i_\loc"'] & \arrow[l, shift right, "L"']\CC \arrow[r, shift right, "\Gamma"'] & \CC^{\K\text{-}\tors}\arrow[l, shift right, "i_\tors"'] 
    \end{tikzcd}
    \quad \text{in} \;\; \Mod_{\CC}(\PrL)
    \]
     where the top composite is cofibre sequence, and the bottom composite is a fibre sequence.  There is also a cofibre sequence 
     \[
     \begin{tikzcd}
        \CC^{\K \text{-}\loc} \arrow[r, "i_\loc"] & \CC \arrow[r, "\Lambda"] & \CC^{\K\text{-}\cpl}
    \end{tikzcd} 
    \quad \text{in}\;\; \Mod_{\CC}(\PrL).
    \]
     \item[(b)] The functors $\Gamma\colon \cat C \to \cat C^{\K\text{-}\tors}$ and $L\colon \cat C \to \cat C^{\K\text{-}\loc}$ are smashing (i.e., $\Gamma x \simeq \Gamma \unit \otimes x$ and $Lx \simeq \L \unit \otimes x$ for all $x\in \cat C$).
     \item[(c)] The functors $\Lambda i_{\tors}\colon \CC^{\K\text{-}\tors} \to \CC^{\K\text{-}\cpl}$ and $\Gamma i_{\cpl}\colon \CC^{\K\text{-}\cpl}\to \CC^{\K\text{-}\tors}$ are mutually inverse equivalences. Moreover there are natural equivalences of functors $\Lambda \Gamma \simeq \Lambda$ and $ \Gamma \simeq \Gamma \Lambda$.
     \end{itemize}
\end{Lem}

\begin{proof}
\leavevmode
\begin{itemize}
\item[(a)] All claims are special cases of  \Cref{lem:setuptors}.
\item[(b)] The fact that $L$ is smashing follows from \Cref{lem:functoriality-of-loc}(d) and the statement for $\Gamma$ follows from this and the fibre sequence involving $\Gamma$ and $L$.
\item[(c)] This is a special case of \Cref{lem:cpltors}.
\end{itemize}
\end{proof}

We next address symmetric monoidal structures in our localizations.

\begin{Rem}\label{rem-local}
   We note that $\CC^{\K\text{-}\loc}$ admits a symmetric monoidal structure such that $L \colon \CC \to \CC^{\K\text{-}\loc}$ lifts to a symmetric monoidal functor, see \cite[Proposition 4.8.2.7]{HA}. In fact, \cite[Proposition 4.8.2.10]{HA} shows that the forgetful functor induces a symmetric monoidal equivalence $\Mod_{L(\unit)}(\cat C)\simeq \CC^{\K\text{-}\loc}$. It follows that canonically $\CC^{\K\text{-}\loc}\in \CAlg(\Pr^L_{\mathrm{st}})$, and that $\CC^{\K\text{-}\loc}$ is compactly generated by dualizable objects. It is rigidly-compactly generated if $\CC$ is so.
\end{Rem}

\begin{Rem}\label{cpl-stable-hom-theory}
    There is an essentially unique symmetric monoidal structure on 
    $\cat C^{\K\text{-}\cpl}$ such that the completion functor $\Lambda \colon \cat C \to\CC^{\K\text{-}\cpl}$ lifts to a symmetric monoidal functor.
    To see this, using  \cite[Proposition 2.2.1.9]{HA}, we have to check that if a map $\alpha$ in $\cat C$ is such that $\Lambda(\alpha)$ is an equivalence, then also $\Lambda(\alpha\otimes\id_X)$ is an equivalence for every $X\in\cat C$. Since the fiber $\mathrm{fib}(\alpha)$ of $\alpha$ is contained in $\cat C^{\K\text{-}\loc}$ and $\mathrm{fib}(\alpha\otimes\id_X)\simeq \mathrm{fib}(\alpha)\otimes X$, this follows from $\cat C^{\K\text{-}\loc}\subseteq\cat C$ being an ideal, which is clear since $\cat C^{\K\text{-}\loc}=L(\cat C)$, and $L$ is a smashing localization. In particular, we canonically have $\cat C^{\K\text{-}\cpl}\in \CAlg(\Pr^L_{\mathrm{st}})$.
\end{Rem}

We also record how the subcategories of local, torsion and complete objects behave under base-change. 

\begin{Prop}\label{prop-base-change-torsion-local}
    Let $\CC\in \CAlg(\PrL)$ be compactly generated by dualizable objects and let $\K \subseteq \CC^\omega$ a set of compact objects. Consider a map $f \colon \CC \to \DD$ in $\CAlg(\PrL)$ whose right adjoint preserves colimits. Then the base-change along $f$ of the diagram of $\CC$-modules
     \begin{equation}\label{cofibre-fibre}
    \begin{tikzcd}
        \CC^{\K \text{-}\loc} \arrow[r, shift right, "i_\loc"'] & \arrow[l, shift right, "L_{\K}"']\CC \arrow[r, shift right, "\Gamma_{\K}"'] & \CC^{\K\text{-}\tors}\arrow[l, shift right, "i_\tors"'] 
    \end{tikzcd}
    \end{equation}
    identifies with 
    \[
    \begin{tikzcd}
        \DD^{f(\K) \text{-}\loc} \arrow[r, shift right, "i_\loc"'] & \arrow[l, shift right, "L_{f(\K)}"']\DD \arrow[r, shift right, "\Gamma_{f(\K)}"'] & \DD^{f(\K)\text{-}\tors}.\arrow[l, shift right, "i_\tors"'] 
   \end{tikzcd}
    \]
    Furthermore, the functor $\DD \otimes_{\CC}\Lambda_{\K}\colon \DD \to \DD \otimes_{\CC}\CC^{\K\text{-}\cpl}$ canonically identifies with $\Lambda_{f(\K)}\colon \DD \to \DD^{f(\K)\text{-}\cpl}$.
\end{Prop}

\begin{proof}
    Firstly, we claim that $\DD \otimes_{\CC} \CC^{\K\text{-}\loc}\simeq \DD^{f(\K)\text{-}\loc}$.  This follows from the following sequence of equivalences:
    \begin{align*}
    \DD \otimes_{\CC} \CC^{\K\text{-}\loc} & \simeq \DD \otimes_{\CC} \Mod_{L_{\K}\unit}(\CC) \\
    & \simeq \Mod_{fL_{\K}\unit}(\DD)\\
    & \simeq \Mod_{L_{f(\K)}\unit}(\DD)\\
    & \simeq \DD^{f(\K)\text{-}\loc}
    \end{align*}
    where we used the description of local objects given in \Cref{rem-local}, the fact that module categories behave well with respect to base-change \cite[Theorem 4.8.4.6]{HA}, and the equivalence $f L_{\K}\simeq L_{f(\K)} f$ which holds by part (d) of \Cref{lem:functoriality-of-loc}, which is applicable by \Cref{lem:check-internal-adjoint}.

    Next, recall from \Cref{lem-properties-torsion-local} that we have a fibre sequence of $\CC$-linear functors $L_{\K}\to \mathrm{id}_{\CC}\to \Gamma_{\K}$ which base-changes to a fibre sequence of $\DD$-linear functors $\DD\otimes_{\CC} L_{\K}\to \mathrm{id}_{\DD}\to \DD \otimes_{\CC} \Gamma_{\K}$. Since localizations are uniquely determined by its subcategory of local objects, the previous paragraph implies that $\DD\otimes_{\CC}L_{\K}\simeq L_{f(\K)}$. From the fibre sequence of $\DD$-linear functors we also get $\DD \otimes_{\CC}\Gamma_{\K}\simeq \Gamma_{f(\K)}$. As a colocalization is uniquely determined by its subcategory of colocal objects, we also obtain $\DD \otimes_{\CC}\CC^{\K\text{-}\tors}\simeq \DD^{f(\K)\text{-}\tors}$ concluding the first part of the proposition. 
    
    For the final claim we consider the cofibre sequence of $\CC$-modules 
    \[
    \begin{tikzcd}
        \CC^{\K \text{-}\loc} \arrow[r, "i_\loc"] & \CC \arrow[r, "\Lambda_{\K}"] & \CC^{\K\text{-}\cpl}
    \end{tikzcd} 
    \]
    from \Cref{lem-properties-torsion-local}. Base-changing to $\DD$ we obtain a cofibre sequence of $\DD$-modules
    \[
    \DD \otimes_{\CC} \CC^{\K\text{-}\loc} \xrightarrow{\DD \otimes_{\CC}i_{\loc}} \DD \xrightarrow{\DD \otimes_{\CC} \Lambda_{\K}}\DD \otimes_{\CC}\CC^{\K\text{-}\cpl}
    \]
    which identifies $\DD \otimes_{\CC}\Lambda_{\K} \colon  \DD \to \DD \otimes_{\CC}\CC^{\K\text{-}\cpl}$ as a Verdier quotient of $\DD$ by $ \DD \otimes_{\CC} \CC^{\K\text{-}\loc} \simeq \DD^{f(\K)\text{-}\loc}$. Again from \Cref{lem-properties-torsion-local}, the same is true for $\Lambda_{f(\K)}\colon \DD \to \DD^{f(\K)\text{-}\cpl}$. So the final claim follows by uniqueness of Verdier quotients.
\end{proof}

We collect some technical observations into the next lemma.

\begin{Lem}\label{megalemma}
     Let $\CC\in \CAlg(\Pr^L_{\mathrm{st}})$ be compactly generated by dualizable objects. Fix a set of compact objects $\K\subseteq \CC^\omega$ defining corresponding categories of local, complete and torsion objects. Then:
     \begin{itemize}
     \item[(a)] $\cat C^\omega \cap \CC^{\tors}=(\CC^{\tors})^\omega$.
     \item[(b)] $\CC^\omega \cap \CC^{\tors}\subseteq\cat C^{\dual} \cap \CC^{\tors} \subseteq (\CC^{\cpl})^{\dual}$. 
     \item[(c)] $\CC^\omega\cap \CC^{\tors}=(\CC^{\cpl})^{\omega}$.
     \item[(d)] For all $x\in \CC^{\dual}$ and $y\in \CC^{\cpl}$, then $x\otimes y \in \CC^{\cpl}$ (where the tensor product is calculated in $\CC$).
     \item[(e)]  Let $x\in\cat C^\omega \cap \CC^{\tors}, y\in(\CC^{\cpl})^{\dual}$. In this case, $x\otimes y\in \CC^\omega\cap \CC^{\tors}$. 
     \end{itemize}
\end{Lem}

\begin{proof}
\leavevmode
\begin{itemize}
\item[(a)] The inclusion functor $\CC^{\tors} \to \CC$ preserves all colimits as it is a left adjoint. In particular, if $t$ is torsion and $\Map_{\CC}(t,-)$ commutes with all filtered colimits in $\CC$, then it also commutes with all filtered colimits in $\CC^{\tors}$, that is we have $\cat C^\omega \cap \CC^{\tors}\subseteq (\CC^{\tors})^\omega$.
 Let conversely $t\in (\CC^{\tors})^\omega$ be compact as a torsion object, and consider a filtered diagram $I\ni i\mapsto y_i\in\CC$. We then find that 
   \[
   \Map_{\CC}(t, \colim y_i) \simeq \Map_{\CC^{\tors}}(t, \Gamma (\colim y_i))\simeq \Map_{\CC^{\tors}}(t, \colim \Gamma (y_i))
   \]
   using that $\Gamma$ preserves colimits since $i_{\tors}$ is an internal left adjoint. This equivalence fits into a commutative diagram 
   \[
   \begin{tikzcd}
       \Map_{\CC}(t, \colim y_i) \arrow[r,"\sim"] & \Map_{\CC^{\tors}}(t, \colim \Gamma( y_i)) \\
      \colim \Map_{\CC}(t,y_i)\arrow[u] \arrow[r, "\sim"] & \colim \Map_{\CC^{\tors}}(t, \Gamma (y_i))\arrow[u,"\sim"]
   \end{tikzcd}
   \]
   using compactness of $t$ for the right vertical equivalence, and a simple adjunction argument for the bottom equivalence. By the 2-out-of-3 property, we conclude that the left vertical arrow is an equivalence, hence showing $\cat C^\omega \cap \CC^{\tors}\supseteq (\CC^{\tors})^\omega$. This proves part (a). 
   \item[(b)]

    The first containment of part (b) is clear as we always have $\CC^\omega \subseteq \CC^{\dual}$,
    see \Cref{initialterm}. For the second containment let $t\in \CC$ be torsion and dualizable. 
    We first show that $t$ is complete, by checking that $\hom_{\CC}(l,t)\simeq 0$ for all local objects $l$.
    Using that $t$ is dualizable, one checks that the canonical map $\mathbb D l\otimes t\to\hom_{\CC}(l,t)$ is an equivalence, so it suffices 
    to verify that $\mathbb D l \otimes t\simeq 0$. 
    This now follows from the observation that $\mathbb D l \otimes t$ is both torsion (as $t$ is torsion and the category of torsion objects forms an ideal) and local since for all torsion object $s$, we have
    \[
    \Map_{\CC}(s, \mathbb D l \otimes t)\simeq \Map_{\CC}(s \otimes l, t) \simeq 0,
    \]
    using that $s \otimes l\simeq 0$ is both local and torsion. We therefore showed that $t$ is complete which implies that the canonical map $t\to \Lambda t$ is an equivalence. Recall that the completion functor $\Lambda \colon \CC \to \CC^{\cpl}$ is symmetric monoidal so it preserves dualizable objects. It then follows that $t \simeq \Lambda t$ is dualizable in $\CC^{\cpl}$ proving (b). 
    \item[(c)] Recall from \Cref{lem-properties-torsion-local}(c) that the completion functor $\Lambda$ induces an equivalence $\CC^{\tors}\xrightarrow{\sim} \CC^{\cpl}$. This furthermore restricts to an equivalence  between compact objects $\Lambda\colon (\CC^{\tors})^\omega\simeq (\CC^{\cpl})^\omega$. Any $t\in (\CC^{\tors})^\omega=\CC^\omega \cap \CC^{\tors}$ is already complete by part (b) so $t \simeq \Lambda t \in (\CC^{\cpl})^\omega$. This gives the containment $\CC^\omega \cap \CC^\tors\subseteq (\CC^{\cpl})^\omega$. For the converse note again that any object $c\in (\CC^{\cpl})^{\omega}$ is necessarily of the form $c\simeq \Lambda t$ for some $t \in(\CC^{\tors})^{\omega}=\CC^\omega \cap \CC^{\tors}$. But $t$ is already complete by part (b) and so $c\simeq t \in(\CC^{\tors})^{\omega}=\CC^{\omega} \cap \CC^{\tors}$ using part (a). This gives the other containment  $(\CC^{\cpl})^{\omega} \subseteq \CC^\omega \cap \CC^\tors$. 
    \item[(d)] We need to verify that $\Map_{\CC}(l, x \otimes y)=0$ for all local objects $l$. For this we compute
    \[
    \Map_{\CC}(l, x \otimes y)\simeq \Map_{\CC}(l \otimes \mathbb{D}x, y)\simeq \Map_{\CC}(\mathbb{D}x, \hom_{\CC}(l,y)),
    \]
 using that $x$ is dualizable. We conclude by observing $\hom_{\CC}(l,y)=0$, since
 \[ \Map_{\CC}(-,\hom_{\CC}(l,y))\simeq \Map_{\CC}(-\otimes l,y)=0,\]
because $-\otimes l$ is local and $y$ is complete.
 \item[(e)] 
We first observe that by part (d), $x\otimes y$ is complete, and compact in the complete category since
 \[ \Map_{\CC^{\mathrm{cpl}}}(x\otimes y,-)\simeq\Map_{\CC^{\mathrm{cpl}}}(x,\hat{\mathbb D} y\ \hat{\otimes}\, -),\]
using that $y$ is dualizable as a complete object, by assumption. Next we observe that, along with $x$, also the uncompleted tensor product $x\otimes y$ is torsion, so we can compute
\begin{align*}
\Map_{\CC}(x\otimes y, -) & \simeq \Map_{\CC^{\mathrm{tors}}}(x\otimes y, \Gamma (-))\\
                         &\simeq \Map_{\CC^{\mathrm{cpl}}}(\Lambda(x\otimes y),\Lambda\Gamma(-)),
\end{align*}
using \Cref{lem:cpltors} in the second step. Since both $\Gamma$ and $\Lambda$ preserve (filtered) colimits and $\Lambda(x\otimes y)=x\otimes y$ is compact, this shows that $\Map_{\CC}(x\otimes y, -)$ commutes with filtered colimits in $\CC$, i.e. we showed $x\otimes y\in\CC^\omega$. Since we observed previously that $x\otimes y$ is torsion, the proof is complete.
    \end{itemize}
\end{proof}

\section{The fracture square}\label{section:fracture}

In this section we will explain how to reconstruct a rigidly-compactly generated symmetric monoidal $\infty$-category $\CC$ from local and complete data while retaining the symmetric monoidal structure. More precisely, we fix a set of compact objects $\K\subseteq \CC^\omega$, and consider the $\infty$-categories $\CC^{\K\text{-loc}}$ and $\CC^{\K\text{-cpl}}$ of local and complete objects respectively. We then introduce in \Cref{def-C-hat} a new $\infty$-category $\widehat{\CC}_{\K}$  which is built from complete data, and construct a $\CC$-linear internal left adjoint $F \colon \CC \to \widehat{\CC}_{\K}$. Finally in \Cref{thm-pullback} we construct a square 
\[
\begin{tikzcd}
    \CC \arrow[d,"F"'] \arrow[r,"L"] & \CC^{\K\text{-}\loc} \arrow[d,"F^{\loc}"] \\
    \widehat{\CC}_{\K} \arrow[r,"\hat{L}"] & (\widehat{\CC}_{\K})^{F(\K)\text{-}\loc}
\end{tikzcd}
\]
and prove it is a pullback in $\CAlg(\PrL)$. 

We start off with the following definition which is the main novelty in our approach to glueing.

\begin{Def}\label{def-C-hat}
    Let $\CC\in \CAlg(\Pr^L_{\mathrm{st}})$ be rigidly-compactly generated and $\K\subseteq \CC^\omega$ a set of compact objects. We set 
    \[
    \widehat{\CC}_{\K}:= \Ind ((\CC^{\K\text{-}\cpl})^{\dual}).
    \]
\end{Def}

Our newly defined $\widehat{\CC}_{\K}$ supports the following structure.

\begin{Rem}
    Recall from \Cref{cpl-stable-hom-theory}  that $\CC^{\K\text{-}\cpl}\in \CAlg(\Pr^L_{\mathrm{st}})$. It then follows from ~\cite[Lemma 2.5]{NaumannPol} that its subcategory of dualizable objects admits the structure of a $2$-ring, and so the discussion after the cited results shows that $\widehat{\CC}_{\K} \in \CAlg(\Pr^L_{\mathrm{st}})$. By construction $\widehat{\CC}_{\K}$ is compactly generated by $(\CC^{\K\text{-}\cpl})^{\dual}$ and $(\widehat{\CC}_{\K})^\omega=(\CC^{\K\text{-}\cpl})^{\dual}$. In fact, $\widehat{\CC}_{\K}$ is rigidly-compactly generated as it is compactly generated by dualizable objects and the unit object is compact.
\end{Rem}

We can now start to set up our fracture square for reconstructing $\CC$.

\begin{Cons}
    The completion functor $\Lambda \colon \CC \to \CC^{\K\text{-}\cpl}$ is symmetric monoidal so it restricts to dualizable objects $\Lambda^{\mathrm{dual}}:\CC^{\mathrm{dual}}\to (\CC^{\K\text{-}\mathrm{cpl}})^{\mathrm{dual}}$, and hence induces a colimit-preserving symmetric monoidal functor $F:=\Ind(\Lambda^{\mathrm{dual}})\colon \CC \to   \widehat{\CC}_{\K}$, using that $\cat C\simeq\Ind(\cat C^{\dual})$, by rigidity.
\end{Cons}

\begin{Lem}\label{lem:F-is-internal}
The functor
\[ 
F\colon \CC\to\widehat{\CC}_{\K}
\]
is an internal left adjoint with respect to the $\CC$-linear structure on $\widehat{\CC}_{\K}$ given by the symmetric monoidal functor $F$.
\end{Lem}

\begin{proof}
Since $\CC$ is generated by dualizables, by \Cref{lem:check-internal-adjoint} it suffices to see that the right adjoint of $F$ preserves colimits. This is well-known to follow from $F$ preserving compact objects, which is clear by construction.
\end{proof}

\begin{Rem}\label{rem-formulas-FandFR}
    We briefly discuss a more concrete description of the functor $F$ and of its right adjoint $F^R$. By definition of $F$, we have a commutative diagram 
    \begin{equation}\label{square-F}
    \begin{tikzcd}
        \CC \arrow[r,"F"] & \widehat{\CC}_{\K} \\
        \CC^\omega \arrow[u, hook] \arrow[r,"\Lambda"] & (\CC^{\K\text{-}\cpl})^{\dual}\arrow[u,"Y"', hook]
    \end{tikzcd}
    \end{equation}
    where the right vertical map is the Yoneda embedding. Thus if we write $x\in \CC$ as a filtered colimits of $x_\alpha\in \CC^\omega$, then 
    \begin{equation*}\label{eq-F-formula}
        F(x)=F(\colim_\alpha x_\alpha)=\colim_\alpha F(x_\alpha)=\colim_{\alpha}Y\Lambda x_\alpha.
    \end{equation*}
    For the right adjoint $F^R$, we claim there is a commutative diagram 
    \begin{equation}\label{square-FR}
    \begin{tikzcd}
        \widehat{\CC}_\K \arrow[r,"F^R"] & \CC\\
        (\CC^{\K\text{-}\cpl})^{\dual}\arrow[u,"Y", hook] \arrow[ur, "i_{\cpl}"'] &.
    \end{tikzcd}
    \end{equation}
    To verify this pick $z \in (\CC^{\K\text{-}\cpl})^{\dual}$ and $x\in \CC$, and write $x$ as a filtered colimits of $x_\alpha\in \CC^\omega$. Then 
    \begin{align*}
        \Map_{\CC}(x, F^R(Y(z))) & \simeq \Map_{\widehat{\CC}_\K}(Fx, Y(z)) \\
                              & \simeq \lim_\alpha \Map_{\widehat{\CC}_\K}(Y\Lambda x_\alpha, Y(z))\\
                              & \simeq \lim_\alpha \Map_{\CC^{\K \text{-}\cpl}}(\Lambda x_\alpha, z) \\
                              & \simeq \lim_\alpha \Map_{\CC}(x_\alpha, i_\cpl z)\\
                              & \simeq \Map_{\CC}(x, i_\cpl z).
    \end{align*}
    Hence if we write $z \in \widehat{\CC}_{\K}$ as a filtered colimit of $Y(z_\beta)$ for $z_\beta \in (\CC^{\K\text{-}\cpl})^{\dual}$, then 
    \[
    F^R(z)=F^R(\colim_\beta Y(z_\beta))=\colim_\beta i_\cpl( z_\beta).
    \]
\end{Rem}

\begin{Cor}\label{cor-propertiesofF}
We have a canonical commutative square of internal left adjoints
 \begin{equation}\label{supersquare}
     \begin{tikzcd}
         \CC \arrow[r, "L"] \arrow[d,"F"'] & \CC^{\K\text{-}\loc} \arrow[d,"F^{\mathrm{loc}}"]\\
         \widehat{\CC}_{\K} \arrow[r,"\hat{L}"] & (\widehat{\CC}_\K)^{F(\K)\text{-}\loc},
     \end{tikzcd}
    \end{equation}
in which $L$ and $\hat{L}$ are the respective localization functors with respect to $\K$ and $F(\K)$. 
Furthermore:
\begin{enumerate}
\item $F$ sends $\K$-local objects to $F(\K)$-local objects.
\item $L$ and $\hat{L}$ are smashing;
\item The Beck-Chevalley map $LF^R\xrightarrow{\sim}(F^{\mathrm{loc}})^R\hat{L}$ is an equivalence.
\item The Beck-Chevalley map $FL^R \xrightarrow{\sim} \hat{L}^R F^\loc$ is an equivalence.
\end{enumerate}
\end{Cor}

\begin{proof}
The four items follow from \Cref{lem:functoriality-of-loc} which is applicable by \Cref{lem:F-is-internal}. 
\end{proof}

\begin{Rem}\label{rem-BC}
    The commutativity of the square in \Cref{cor-propertiesofF} is witnessed by a natural transformation $\alpha \colon F^{\loc} L \simeq \hat{L}F$. Passing to the total mate we get a natural transformation $\overline{\alpha}\colon F^R\hat{L}^R\simeq L^R(F^\loc)^R$ which has an associated left Beck-Chevalley map $FL^R \to \hat{L}^R F^\loc$. We claim that this latter map is also an equivalence. Indeed the natural transformation $\overline{\alpha}$ can be obtained by first considering the horizontal right Beck-Chevalley transformation of $\alpha$, and the passing to the vertical right Beck-Chevalley transformation. Using that passage to the left and right Beck-Chevalley transformations are inverse constructions, we deduce that the vertical left Beck-Chevally of $\overline{\alpha}$ is homotopic to the horizontal right Beck-Chevalley transformation of $\alpha$, which is an equivalence by \Cref{cor-propertiesofF}(d).
\end{Rem}

Our main theorem, then, is the following.

\begin{Thm}\label{thm-pullback}
    Let $\CC\in \CAlg(\Pr^L_{\mathrm{st}})$ be rigidly-compactly generated. Then the square (\ref{supersquare}) is a pullback diagram in $\CAlg(\Pr^L_{\mathrm{st}})$.
\end{Thm}

For the proof, we first review some facts about pullbacks.

\begin{Rem}\label{rem:explicit_adjoint}
Assume we have a commutative square in $\CAlg(\Pr^L_{\mathrm{st}})$ of the following shape:
\[
\begin{tikzcd}
    \CC_0 \arrow[r,"F_{01}"] \arrow[d,"F_{02}"'] & \CC_1 \arrow[d,"F_{13}"] \\
    \CC_2 \arrow[r,"F_{23}"] & \CC_3.
\end{tikzcd}
\]
By definition, the functors $F_{ij}$ admit right adjoints $F_{ij}^R$. The universal property of the pullback determines a functor $G \colon \CC_0 \to \CC_1 \times_{\CC_3} \CC_2$, which sends an object $X\in \CC_0$ to the triple $(F_{01}(X), F_{02}(X), \eta \colon F_{12}F_{01}(X)\simeq F_{23}F_{02}(X))$ where the equivalence $\eta$ is the one witnessing the commutativity of the above square. The functor $G$ admits a right adjoint $G^R$ which sends a triple $(X_1, X_2, \eta\colon F_{13}X_1 \simeq F_{23}X_2)\in \CC_1 \times_{\CC_3}\CC_2$ to the pullback in $\CC_0$ of the following diagram 
\[
\begin{tikzcd}
                & F_{01}^RX_1 \arrow[d] \\
 F_{02}^R X_2\arrow[r] & F_{02}^R F_{23}^R F_{23}X_2\simeq F_{01}^R F_{13}^RF_{13}X_1
\end{tikzcd}
\]
where the two arrows are induced by the unit maps of the adjunctions $(F_{23}, F_{23}^R)$ and $(F_{13}, F_{13}^R)$, and the bottom right equivalence is given by $F_{01}^RF_{13}^R(\eta)$ and the equivalence $F_{02}^R F_{23}^R\simeq F_{01}^R F_{13}^R$ (which follows from the commutativity of the square by passing to right adjoints). This description of $G^R$ is a special case of~\cite[Theorem 5.5]{HY2017}. By the cited result, we can also describe the unit and counit maps of the adjunction $(G,G^R)$  explicitly as follows. 
For any $X\in \CC_0$, the unit map $\eta_X \colon X \to G^R GX$ fits into a commutative diagram 
    \[
    \begin{tikzcd}
        X \arrow[dr,dotted, "\eta_X"]\arrow[rrd, bend left, "\mathrm{unit}"] \arrow[ddr, bend right, "\mathrm{unit}"']& & \\
                               & G^RG X \arrow[d] \arrow[r] & F_{01}^R F_{01}X \arrow[d]\\
                               & F_{02}^R F_{02}X \arrow[r] & (F_{23}F_{02})^R(F_{23}F_{02})X \simeq (F_{13}F_{01})^R(F_{13}F_{01})X
    \end{tikzcd}
    \]
    where the small square is a pullback. For any object $Z=(X_1,X_2, \eta \colon F_{13}X_1 \simeq F_{23}X_2)\in  \CC_1 \times_{\CC_3} \CC_2$, 
    the counit map $\epsilon_Z \colon GG^RZ \to Z$ has components given by the compositions
\[ f:F_{01}(G^RZ)\xrightarrow{F_{01}(\pi_1)} F_{01}F_{01}^R X_1 \to X_1 \mbox{ and }\]
\[ g:F_{02}(G^RZ) \xrightarrow{F_{02}(\pi_2)} F_{02}F_{02}^RX_2 \to X_2, \]
with the second maps being the counits.
\end{Rem}

The next two lemmas serve to sufficiently control the difference between $\CC$ and
$\widehat{\CC}$.

\begin{Lem}\label{lem-torsunambig}
    Let $x\in \CC^\omega\cap \CC^{\K\text{-}\tors}\subseteq(\CC^{\K\text{-}\cpl})^{\mathrm{dual}}\subseteq \CC^{\K\text{-}\cpl}\subseteq\CC$ (see \Cref{megalemma}(b)),  
    and consider the adjunctions 
    \[ 
    \Lambda \colon \CC\rightleftarrows \CC^{\K \text{-}\cpl}: i_{\cpl}\,\,\,\mbox{ and }\,\,\, F\colon \CC\rightleftarrows \widehat{\CC}_{\K}: F^R,
    \]
    as well as the restricted Yoneda embedding
    \[
    Y\colon (\CC^{\K \text{-}\cpl})^{\mathrm{dual}}\hookrightarrow \widehat{\CC}_{\K}=\mathrm{Ind}((\CC^{\K \text{-}\cpl})^{\mathrm{dual}}).
    \]
    Then the following unit and counit maps are equivalences:
    \begin{enumerate}
        \item $x\xrightarrow{\sim} i_{\mathrm{cpl}}\Lambda(x)$ and $\Lambda i_{\mathrm{cpl}}(x)\xrightarrow{\sim} x$.
        \item $x\xrightarrow{\sim} F^RF(x)$ and $FF^RY(x)\xrightarrow{\sim} Y(x)$.
    \end{enumerate}
    \end{Lem}

\begin{proof}
Claim (a) is trivial because $(\Lambda,i_{\mathrm{cpl}})$ is a localization for which $x$ is a local object by assumption.
To prove (b), we first observe that generally, if $f\colon A\rightleftarrows B:g$ is an adjunction, then the unit map $a\to gfa$ is an equivalence if and only if for all $a'\in A$, the map $\Map_A(a',a)\to \Map_B(f(a'),f(a))$ induced by $f$ is an equivalence.\\ 

For the unit map in (b), we thus need to show that the map 
    \[ 
    \Map_{\CC}(y,x) \to \Map_{\widehat{\CC}_{\K}}(F(y),F(x))
    \]
    is an equivalence for all $x\in\cat C^\omega\cap \cat C^{\K \text{-}\tors}$ and $y\in\cat C$. Since $\cat C$ is compactly generated and $F$ preserves colimits, we may assume that $y\in \CC^{\omega}$. In this case, the map under consideration identifies by (\ref{square-F}) with the map
    \[ 
    \Map_{\CC}(y,x) \to \Map_{\widehat{\CC}_{\K}}(Y\Lambda(y), Y\Lambda(x))\simeq\Map_{\CC^{\K \text{-}\mathrm{cpl},\mathrm{dual}}}(\Lambda(y),\Lambda(x)),
    \]
    which is an equivalence because we assumed that $x\simeq\Lambda(x)$ is complete.
    To treat the counit in part (b), we check the dual version of our initial observation, namely that for all $x\in\cat C^\omega\cap \cat C^{\K \text{-}\tors}$, and $y\in\widehat{\cat C}_\K$, the map induced by $F^R$
    \begin{equation}\label{eq:test-map}
    \Map_{\widehat{\CC}_{\K}}(Y(x),y) \to \Map_{\CC}(F^R(Y(x)), F^R(y))
    \end{equation}
    is an equivalence.
    Since $F^R$ preserves colimits by \Cref{lem:F-is-internal}, we may assume without loss of generality that $y=Y(z)$ for some $z\in(\CC^{\K \text{-}\cpl})^{\dual}$. Using the commutative diagram (\ref{square-FR}), the map \Cref{eq:test-map} identifies with the map
    \[ 
    \Map_{\CC^{\K \text{-}\cpl}}(x,z) \simeq \Map_{\widehat{\CC}_\K}(Y(x),Y(z))\to \Map_{\CC}(i_{\mathrm{cpl}}(x),i_{\mathrm{cpl}}(z))
    \]
    induced by $i_\cpl$, and therefore it is an equivalence.
\end{proof}

\begin{Lem}\label{lem:control-fiber}
\leavevmode 
\begin{enumerate}
    \item For every $x\in\CC^\omega\cap\CC^{\K \text{-}\mathrm{tors}}$ and $z\in\widehat{\CC}_{\K}$, the map induced by the counit map 
\[
    F(x)\otimes FF^R(z) \to  F(x)\otimes z
\] 
is an equivalence.
\item For all $t\in (\widehat{\CC}_{\K})^{F({\mathcal K})\text{-}\mathrm{tors}}$, the counit map $FF^R(t) \to t$ is an equivalence.
\end{enumerate}
\end{Lem}

\begin{proof}
\leavevmode
\begin{enumerate}
\item Since $F$ and $F^R$ preserve colimits, we can assume that $z=Y(y)$ for some $y\in(\CC^{\K\text{-}\mathrm{cpl}})^{\dual}$. We then consider each side separately. 
We have
\[ 
F(x)=Y(\Lambda(x))\simeq Y(x),
\]
using (\ref{square-F}) and the fact that $x$ is complete. So the right hand side identifies with 
\[ 
Y(x)\otimes Y(y)\simeq Y(x\otimes y).
\]
For the left hand side, we recall from (\ref{square-FR}) that $F^R(Y(y))\simeq i_{\mathrm{cpl}}(y)\simeq y$, so using that $F$ is symmetric monoidal, the left hand side becomes equivalent to $F(x\otimes y)$. Now, the map we want to see is an equivalence is of the form
\[ 
F(x\otimes y)\to Y(x\otimes y).
\]
Since by \Cref{megalemma}(e) we have $x\otimes y\in\CC^\omega\cap\CC^{\K \text{-}\mathrm{tors}}$, the source of the map simplifies to $Y\Lambda(x \otimes y)$, and so the required map is an equivalence by the completeness of $x\otimes y$.
\item Since both $F$ and $F^R$ preserve colimits, we can assume that
\[ 
t = F(x)\otimes Y(y)
\]
for some $x\in\CC^\omega\cap\,\CC^{{\mathcal K}\text{-}\mathrm{tors}}$ and some $y\in(\CC^{\K \text{-}\mathrm{cpl}})^{\dual}$. We want to see that in this case, the counit map
\[ 
FF^R(F(x)\otimes Y(y))\to F(x)\otimes Y(y) 
\]
is an equivalence. Using that $F^R$ and $F$ are $\CC$-linear, this map is equivalent to the tensor product of $F(x)$ with the counit map at $Y(y)$, and hence an equivalence by part (a).
\end{enumerate}
\end{proof}

\begin{proof}[Proof of \Cref{thm-pullback}]
    Recall that we use the superscript $R$ to denote the right adjoint to a functor, so that $F^R$ is the right adjoint to $F$. 
    By Remark \ref{rem:explicit_adjoint}, there is an adjunction 
    \[
    G \colon \CC \rightleftarrows \widehat{\CC}_\K \times_{(\widehat{\CC}_\K)^{F(\K)\text{-}\loc}} \CC^{\K\text{-}\loc}: G^R
    \]
    whose effect on objects was recalled there. We claim that the unit and counit maps of this adjunction are equivalences. 

    We start off by checking that the unit map $\eta_x \colon x \to G^R Gx$ is an equivalence for all $x\in \CC$. We recorded in \Cref{rem:explicit_adjoint} that the unit map $\eta_x$ fits into a commutative diagram 
    \begin{equation}\label{unit}
    \begin{tikzcd}
        x \arrow[dr,dotted, "\eta_x"]\arrow[rrd, bend left, "\mathrm{unit}"] \arrow[ddr, bend right, "\mathrm{unit}"']& & \\
                               & G^RG x \arrow[d] \arrow[r] & L^R Lx \arrow[d]\\
                               & F^R Fx \arrow[r] & (\hat{L}F)^R(\hat{L}F)x \simeq (F^{\mathrm{loc}}L)^R(F^{\mathrm{loc}}L)x
    \end{tikzcd}
    \end{equation}
    where the small square is a pullback. Using the triangle $\Gamma x \to x \to Lx$ and the five lemma, it is enough to show that the unit map $\eta_x$ is an equivalence for $x$ torsion or local. We therefore distinguish two cases:
    \begin{itemize}
        \item If $x$ is torsion, then the top right and bottom right vertices of the diagram~(\ref{unit}) are zero as $Lx \simeq 0$. It follows that $\eta_x$ is an equivalence if and only if the unit map $x\to F^R Fx$ is so. Since $F^R$ preserves colimits, We can assume that $x$ compact in $\CC$ so that $x\in \CC^{\K \text{-}\tors}\cap \CC^\omega\subseteq (\CC^{\K\text{-}\cpl})^{\dual}$ by \Cref{megalemma}(b). In this case the unit map is an equivalence by \Cref{lem-torsunambig}(b). 
        \item If $x$ is local, then the unit map $x \to L^RL x$ is an equivalence. The 2-out-of-3 property for equivalences shows that $\eta_x$ is an equivalence if the lower horizontal map in (\ref{unit}), namely  
        \[ 
        \left( F^R Fx \to F^R(\hat{L}^R\hat{L}Fx) \right) = F^R\left(Fx\to \hat{L}^R\hat{L}Fx\right)
        \] 
        is an equivalence. Indeed, already the map $Fx\to \hat{L}^R\hat{L}Fx$ is an equivalence because $Fx$ is local by \Cref{cor-propertiesofF}(a).    
 \end{itemize}
Next, we check that the counit map $\epsilon_z \colon GG^Rz \to z$ is an equivalence for all objects $z=(x,y, \hat{L}x \simeq F^{\mathrm{loc}}y)\in \widehat{\CC}_\K  \times_{(\widehat{\CC}_\K)^{F(\K)\text{-}\loc}} \CC^{\K\text{-}\loc}$. Recall from Remark \ref{rem:explicit_adjoint} that we have a pullback diagram in $\cat C$
\begin{equation}\label{eq:GZ}
\begin{tikzcd}
    G^Rz \arrow[rrr,"\pi_L"] \arrow[d,"\pi_F"] & & & L^Ry  \arrow[d,"L^R(y\to (F^{\mathrm{loc}})^R F^{\mathrm{loc}}y)"] \\
    F^Rx \arrow[rrr,"F^R(x\to \hat{L}^R\hat{L}x)"] & & & F^R\hat{L}^R\hat{L}x\simeq L^R(F^{\mathrm{loc}})^R F^{\mathrm{loc}} y.
\end{tikzcd}
\end{equation}

Also, the two components of $\epsilon_z$ which we need to see are equivalences, are the compositions
\[ 
f\colon L(G^Rz)\xrightarrow{L(\pi_L)} LL^R y \to y \mbox{ and }
\]
\[ 
g\colon F(G^Rz) \xrightarrow{F(\pi_F)} FF^Rx \to x, 
\]
with the second maps being the counits.
To see that $f$ is an equivalence, note that the counit $LL^Ry\to y$ is an equivalence because $L$ is a localization, hence $f$ is an equivalence if and only if $L(\pi_L)$ is. Now $L(\pi_L)$ is a pullback of 
\[ 
LF^R(x\to \hat{L}^R\hat{L}x) \simeq (F^{\mathrm{loc}})^R(\hat{L}x \to \hat{L}\hat{L}^R\hat{L}x),
\]
using that $LF^R\simeq (F^{\mathrm{loc}})^R\hat{L}$ by \Cref{cor-propertiesofF}(c).
This latter map is an equivalence because $\hat{L}x\to \hat{L}\hat{L}^R\hat{L}x$ is, $\hat{L}$ being a localization. This shows that $f$ is an equivalence. Seeing that $g$ is an equivalence is a little more involved.
By applying $F$ to the rotation of \Cref{eq:GZ} and post composing with counits, we have $g$ appearing as the top horizontal composition in the following diagram:
\begin{equation}\label{diagramma}
\begin{tikzcd}
    g\colon  FG^Rz \arrow[rrr,"F(\pi_F)"] \arrow[d,"F(\pi_L)"'] & & & FF^Rx  \arrow[d,"FF^R(x\to \hat{L}^R\hat{L}x)"] \ar[rr] & & x\arrow[d]\\
    FL^Ry \arrow[rrr,"FL^R(y\to(F^{\mathrm{loc}})^RF^{\mathrm{loc}}y)"] & & & FF^R\hat{L}^R\hat{L}x \arrow[rr]& &  \hat{L}^R\hat{L}x.
\end{tikzcd}
\end{equation}
We start by claiming that the bottom horizontal composition is an equivalence. Unravelling the definitions, we see that the required map is given by the composite 
\[
FL^R y \xrightarrow{\mathrm{unit}} FL^R (F^\loc)^R F^\loc y\xrightarrow[\sim]{\mathrm{ident.}} FL^R (F^\loc)^R\hat{L}x \xrightarrow[\sim]{\mathrm{comm.}} FF^R \hat{L}^R \hat{L}x \xrightarrow{\mathrm{counit}} \hat{L}^R \hat{L}x
\]
where we used the identification $F^\loc y \simeq \hat{L}x$ part of the data of $z$, and the equivalence $L^R (F^\loc)^R \simeq F^R \hat{L}^R$ which is obtained by passing to right adjoints the natural equivalence witnessing the commutativity of the diagram involving $L$ and $F$. Since this latter equivalence is natural, we can exchange it with the identification map and obtain the composite
\[
FL^R y \xrightarrow{\mathrm{unit}} FL^R (F^\loc)^R F^\loc y\xrightarrow[\sim]{\mathrm{comm.}} FF^R \hat{L}^R (F^\loc)y \xrightarrow[\sim]{\mathrm{ident.}} FF^R \hat{L}^R \hat{L}x \xrightarrow{\mathrm{counit}} \hat{L}^R \hat{L}x.
\]
Next, using the naturality of the counit map we can exchange the last two maps and obtain
\[
FL^R y \xrightarrow{\mathrm{unit}} FL^R (F^\loc)^R F^\loc y\xrightarrow[\sim]{\mathrm{comm.}} FF^R \hat{L}^R (F^\loc)y \xrightarrow{\mathrm{counit}} \hat{L}^R F^\loc y \xrightarrow[\sim]{\mathrm{ident.}} \hat{L}^R \hat{L}x.
\]
The composite of the first three maps is precisely the BC map of $L^R (F^\loc)^R \simeq F^R \hat{L}^R$ which is an equivalence by \Cref{rem-BC}. This proves the claim as all the maps involved are equivalences. 

Going back at the diagram (\ref{diagramma}); we want to show that $g$ is an equivalence. Using the bottom composite in an equivalence and that the left square in the above diagram is a pull-back, an elementary argument which we leave to the reader shows that $g$ is an equivalence if and only if the right hand square, namely
\[\begin{tikzcd}
    FF^Rx  \arrow[d,"FF^R(x\to \hat{L}^R\hat{L}x)"'] \ar[rr] & & x\arrow[d]\\
    FF^R\hat{L}^R\hat{L}x \arrow[rr]& &  \hat{L}^R\hat{L}x
\end{tikzcd}
\]
is a pullback. To see this, passing to horizontal fibers reduces us to showing that the counit map
\[ 
FF^Rt\to t
\]
is an equivalence for every $t\in(\widehat{\cat C}_\K)^{F(\K)\text{-}\tors}$.
This was already estabished in \Cref{lem:control-fiber}(b) above.
\end{proof}

Recall that a $2$-ring is an essentially small, idempotent complete, stable and symmetric monoidal $\infty$-category. The collection of $2$-rings and symmetric monoidal exact functors between them can be organized into an $\infty$-category. This category admits limits and these are preserved by the forgetful functor to $\Cat_\infty$, see discussion in \cite[Subsection 2.1]{Mathew2016}.
\begin{Not}
For a $2$-ring $\BB$ and a set of objects $\K\subseteq \BB$, we write $\thickt{\K}$ for the thick tensor ideal generated by $\K$ in $\BB$.
Similarly, given $\CC\in \CAlg(\Pr_{\mathrm{st}}^L)$ and a set of objects $\K\subseteq\CC$, we write $\loct{\K}$ for the localizing ideal generated by $\K$ in $\CC$.
\end{Not}

\begin{Cor}\label{pullback-small}
    Let $\BB$ be a $2$-ring consisting of dualizable objects and $\K$ a set of objects of $\BB$. Then there is a pullback of $2$-rings
    \[
    \begin{tikzcd}
        \BB \arrow[r]\arrow[d,"\Lambda"'] & (\BB/\thickt{\K})^\natural\arrow[d]\\
        \Ind(\BB)^{\K\text{-}\mathrm {cpl},\mathrm{dual}}\arrow[r] & \left( \Ind(\BB)^{\K\text{-}\mathrm {cpl},\mathrm{dual}}/ \thickt{\Lambda\K}\right)^\natural.
    \end{tikzcd}
    \]
    where the horizontal arrows are the functors defining the Verdier quotients, the right vertical arrow is induced by $\Lambda$, and $(-)^\natural$ denotes the idempotent completion functor.
  \end{Cor}

\begin{proof}
    We apply~\Cref{thm-pullback} to $\Ind(\BB)$, which is rigidly-compactly generated by our assumption on $\BB$, and obtain a pullback square 
    \[
      \begin{tikzcd}
         \Ind(\BB) \arrow[r, "L"] \arrow[d,"F"'] & \Ind(\BB)^{\K\text{-}\loc} \arrow[d,"F^{\mathrm{loc}}"]\\
         \widehat{\Ind(\BB)}_\K \arrow[r,"\hat{L}"] & (\widehat{\Ind(\BB)}_\K)^{F(\K)\text{-}\loc}.
     \end{tikzcd}
    \]
    Note that all the above categories are rigidly-compactly generated so the functors $(-)^{\dual}$ and $(-)^\omega$ agree in this case. Applying the functor $(-)^{\dual}\simeq (-)^\omega$ to the above square we obtain another pullback square whose left most entries agree with the ones of the corollary. We only need to identify the two categories on the right hand entries of the square as in the corollary.  
    
    By definition of $\K$-local and $\K$-torsion objects, there is a fibre sequence
    \[
    \loct{\K} \to \Ind(\BB) \xrightarrow{L} \Ind(\BB)^{\K\text{-}\loc}.
    \]
    In other words, we can identify the category of local objects as a Verdier quotient of $\Ind(\BB)$ by the subcategory of torsion objects. By the Neeman-Thomason localization theorem \cite[Theorem 2.1]{Neeman92}, we can described the compact objects in the local category as
    \[
    (\Ind(\BB)^{\K\text{-}\loc})^\omega\simeq (\BB /\thickt{\K})^\natural
    \]
    and under this identification the localization functor $L$ agrees with the canonical functor to the Verdier quotient. This identifies the top right entry of the square and the top horizontal arrow as in the corollary. Applying the same argument to $(\widehat{\Ind(\BB)}_\K)^{F(\K)\text{-}\loc}$ let us identify the bottom right entry as in the corollary.   
\end{proof}

\section{Equivariant preliminaries}\label{section:EQ1}
In this section we record some useful results on categories with an action of a finite group. In the next section we will apply these results and the main result of this paper in the setting of equivariant homotopy theory to obtain a variant of the isotropy separation sequence. We start off with the following:

\begin{Def}
    Give a group $G$ and $\cat C \in \Fun(BG, \Cat_\infty)$, we can form 
    \[
    \cat C^{hG} := \lim_{BG}\cat C \in\Cat_\infty.
    \]
\end{Def}

We now record the main example for this paper.

\begin{Exa}\label{def-hG-modA}
Let $G$ be a group and consider $A\in \Fun(BG, \CAlg)$. The category of $A$-modules inherits a canonical $G$-action, yielding a functor
\[
\Mod(A)\colon BG \to \PrL.
\]
We can then form $\Mod(A)^{hG}$
which we think of as the category of $A$-modules endowed with a semi-linear $G$-action. 
We note that the functor $\Mod(A)$ naturally lifts to $\CAlg(\PrL)$ and so canonically $\Mod(A)^{hG}\in \CAlg(\PrL)$. We write $\Perf(A)^{hG}$ for the full subcategory of dualizable objects in $\Mod(A)^{hG}$.
\end{Exa}

We next show how adjunctions are preserved under passage to invariants under a group action. This is a special case of the following more general exposition.\\
Given a simplicial set $S$, recall from \cite[Definition 4.7.4.16]{HA}, the \emph{non-full} subcategory
\begin{equation}\label{eq:left-adjointables}
\Fun^{\mathrm{LAd}}(S,\Cat_\infty)\subseteq\Fun(S,\Cat_\infty)
\end{equation}
with objects those functors whose transition maps admit left adjoints, and morphisms those natural transformations which yield left adjointable squares. There is an analogous definition of 
$\Fun^{\mathrm{RAd}}(S, \Cat_\infty)$ using right adjointability.

\begin{Lem}\label{lem:groupoid-core}
Let $S$ be a simplicial set. Applying the groupoid cores to (\ref{eq:left-adjointables}), we obtain a fully faithful inclusion
\[ 
\Fun^{\mathrm{LAd}}(S,\Cat_\infty)^{\simeq}\subseteq\Fun(S,\Cat_\infty)^{\simeq}.
\]
The analogous statement holds for $\Fun^{\mathrm{RAd}}(S, \Cat_\infty)$ as well.
\end{Lem}

\begin{proof}
We treat the case of left adjointable squares and leave the analogous case to the reader. The inclusion is clearly faithful, so the claim is that it is full, i.e.
given $F,G\in\Fun^{\mathrm{LAd}}(S, \Cat_\infty)$ and an equivalence $\alpha\colon F\xrightarrow{\sim}
G$ in $\Fun(S, \Cat_\infty)$, we claim that we have $\alpha\in\Fun^{\mathrm{LAd}}(S, \Cat_\infty)$.
By definition, this means that for every edge $\gamma\colon s\to s'$ in $S$, the following square
 \[
    \begin{tikzcd}
        {F(s)} &  {F(s')} \\
        {G(s)} & {G(s').}
        \arrow["\alpha(s')",from=1-2, to=2-2]
        \arrow["F(\gamma)", from=1-1, to=1-2]
        \arrow["G(\gamma)"', from=2-1, to=2-2]
        \arrow["\alpha(s)"', from=1-1, to=2-1]
        \arrow["\nu_\gamma", shorten <=7pt, shorten >=7pt, Rightarrow, from=1-2, to=2-1]
    \end{tikzcd}
    \]
    is (horizontally) left adjointable. By assumption, $F(\gamma)$ and $G(\gamma)$ admit left adjoints $F^L(\gamma)$ and $G^L(\gamma)$, and there is an equivalence $\nu_\gamma\colon\alpha(s')F(\gamma)\Rightarrow G(\gamma)\alpha(s)$, and we need to see that its Beck-Chevalley transform
$\mathrm{BC}(\nu_\gamma)\colon G^L(\gamma)\alpha(s')\Rightarrow\alpha(s) F^L(\gamma)$ is an equivalence. For this, we consider the whiskering $\mu_\gamma:=\alpha(s')^{-1}\circ\nu_{\gamma}\cdot\alpha(s)^{-1}$, whose inverse is an equivalence 
\[ \mu^{-1}_\gamma\colon \alpha(s')^{-1}G(\gamma)\Rightarrow F(\gamma)\alpha(s)^{-1}.
\]
By construction we have a commutative diagram
\[
    \begin{tikzcd}[column sep =1.2 cm, row sep =1.2 cm]
        {G(s)} &  {G(s')} \\
        {F(s)} &  {F(s')} \\
        {G(s)} & {G(s')} \\
        {F(s)} & {F(s')}
        \arrow["\alpha(s')^{-1}",from=1-2, to=2-2]
        \arrow["\alpha(s)^{-1}"', from=1-1, to=2-1]
        \arrow["F(\gamma)", from=2-1, to=2-2]
        \arrow["G(\gamma)", from=1-1, to=1-2]
        \arrow["\mu_\gamma^{-1}", shorten <=12pt, shorten >=12pt, Rightarrow, from=1-2, to=2-1]
        \arrow["\alpha(s)"', from=2-1, to=3-1]
        \arrow["\alpha(s)^{-1}"', from=3-1, to=4-1]
        \arrow["G(\gamma)", from=3-1, to=3-2]
        \arrow["F(\gamma)", from=4-1, to=4-2]
        \arrow["\alpha(s')", from=2-2, to=3-2]
         \arrow["\alpha(s')^{-1}", from=3-2, to=4-2]
           \arrow["\nu_\gamma", shorten <=12pt, shorten >=12pt, Rightarrow, from=2-2, to=3-1]
            \arrow["\mu^{-1}_\gamma", shorten <=12pt, shorten >=12pt, Rightarrow, from=3-2, to=4-1]
    \end{tikzcd}
    \]
    such that if we paste the top two squares we obtain $\mathrm{id} \colon G(\gamma) \Rightarrow G(\gamma)$, and if we paste the bottom two squares we obtain $\mathrm{id}\colon F(\gamma) \Rightarrow F(\gamma)$. Passing to the associated diagram of BC-transformations and using the composition law for BC-maps, we see that $BC(\nu_\gamma)$ admits both a left and a right inverse, and so it is an equivalence. 
\end{proof}

We will use this result through its following corollary.

\begin{Cor}\label{cor:adjoints-of-fixed-points}
Let $G$ be a group and $f\colon \cat C\to \cat D$ a map in $\Fun(BG,\Cat_\infty)$ whose image in $\Cat_\infty$ admits a left (resp. right) adjoint. Then this adjoint is canonically $G$-equivariant, and passage to homotopy fixed-points yields a commutative diagram
\[
\begin{tikzcd}
    \cat C^{hG} \arrow[r,"f^{hG}"] \arrow[d, "\mathrm{fgt}"'] & \cat D^{hG} \arrow[d,"\mathrm{fgt}"]\\
    \cat C \arrow[r, "f"] & \cat D
\end{tikzcd}
\]
which is (horizontally) left (resp. right) adjointable.
\end{Cor}

\begin{proof}
We prove the claims involving left adjoints and leave the analogous proof involving right adjoints to the reader.
The given data can be encoded by a map $BG\to\Fun(\Delta^1,\Cat_\infty)$ which automatically factors through the groupoid core. In fact, it even factors through 
$\Fun^{\mathrm{LAd}}(\Delta^1, \Cat_\infty)^\simeq$, since by \Cref{lem:groupoid-core} this can be checked on the unique object, where it amounts to our assumption on $f$. The claim now follows from the fact that $\Fun^{\mathrm{LAd}}(\Delta^1, \Cat_\infty)\subseteq\Fun(\Delta^1, \Cat_\infty))$ is closed under limits by \cite[Corollary 4.7.4.18]{HA}. 
\end{proof}

\begin{Cons}
    For any subgroup $H \subseteq G$, we have a diagram of adjunctions
\[ \begin{tikzcd}
\Fun(BH,\Cat_\infty) \ar[rr, shift right=0.5 ex, "\mathrm{CoInd}_H^G" '] & & \Fun(BG,\Cat_\infty) \ar[ll, shift right = 0.5 ex, "\mathrm{Res}^G_H" '] \ar[rr, shift right = 0.5 ex, "(-)^{hG}" '] & &  \Cat_\infty \ar[ll, shift right = 0.5 ex, "\mathrm{const}" ' ]
\end{tikzcd}
\]
with left adjoints displayed on top. We note that $(-)^{hG}\circ\mathrm{Coind}^G_H\simeq (-)^{hH}$ which can be deduced from the equivalence $\mathrm{Res}^G_H\circ \mathrm{const}\simeq \mathrm{const}$ by passing to right adjoints. For any $\cat C\in \Fun(BG, \Cat_\infty)$, the unit map of the restriction-coinduction adjunction induces a restriction map
\[
\res_H^G \colon \CC^{hG} \to \mathrm{Coind}^G_H(\mathrm{Res}^G_H(\cat C))^{hG}\simeq \CC^{hH}
\]
\end{Cons}

The main result of this section is the following result. 

\begin{Prop}\label{lem:leftright-adjoint}
Let $G$ be a group and $\cat C\in \Fun(BG, \Cat_\infty)$ whose underlying $\infty$-category is additive. For every subgroup $H \subseteq G$ of finite index, the functor
\[ 
\res_H^G\colon \cat C^{hG}\to \cat C^{hH}
\]
admits a left and a right adjoint, and these two functors are equivalent. Furthermore these adjunctions preserves dualizable objects.
\end{Prop}

\begin{proof}
We apply \Cref{cor:adjoints-of-fixed-points} to the unit map
\[ 
f\colon \cat C\to\mathrm{Coind}^G_H(\mathrm{Res}^G_H(\cat C)).
\]
We observe that $\mathrm{Res}^G_e\mathrm{Coind}^G_H(\mathrm{Res}^G_H(\cat C)) \simeq \prod_{G/H}\cat C$, and that under this identification the map $\mathrm{Res}^G_e(f)$ becomes informally given by $x \mapsto (gx)_{g \in G/H}$ for some choice of representatives of $G/H$ \footnote{Denote by $*_G\in BG$ the unique object and by $i\colon H\subseteq G$ the given inclusion. To see the claim use the formula for the right Kan extension which describe $\mathrm{Res}^G_e\mathrm{Coind}^G_H(X)$ as a limit over the slice category $BH_{*_G/Bi}$. We note that there is an equivalence of discrete categories $BH_{*_G/Bi}\simeq G/H$ under which the canonical map $BH_{*_G/Bi}\to BH$ identifies with the constant map.}.
We also observe that $\mathrm{Res}^G_e(f)$ differs from the diagonal $\Delta \colon \CC \to \prod_{G/H} \CC$ by composition with a non-equivariant equivalence $\prod_{G/H} \CC \simeq \prod_{G/H}\CC$, $(x)_{g\in G/H} \mapsto (g^{-1}x)_{g \in G/H}$ for the same choice of representatives for $G/H$.
Therefore $\mathrm{Res}^G_e(f)$ admits a left (resp. right) adjoint if $\cat C$ admits finite coproducts (resp. finite products), and these are equivalent if $\cat C$ is additive. It then follows that the map 
\[
\res_H^G \colon \cat C^{hG}\xrightarrow{f^{hG}}\mathrm{Coind}^G_H(\mathrm{Res}^G_H(\cat C))^{hG}\simeq \cat C^{hH}
\]
admits a left and right adjoint and that they agree. The last claim about dualizability follows from the fact that finite sums of dualizables are dualizable and that dualizability is detected on the vertices of a limit diagram.
\end{proof}

\begin{Not}\label{not-induced}
    Let $G$ be a group and $A\in \Fun(BG, \CAlg)$. For every subgroup $H \subseteq G$ of finite index, we denote the left and right adjoint of the restriction functor
\[ 
\res_H^G\colon \Mod(A)^{hG}\to \Mod(A)^{hH}
\]
respectively by $G \circledast_H (-)$ and $(-)^{\circledast_H G}$. In the special case where $H=e$, we will simply write $G \circledast (-)$ and $(-)^{\circledast G}$.
\end{Not}
We finish this section with the following result. 

\begin{Lem} 
Let $G$ be a group, $\cat C$ and $\cat D$ be $\infty$-categories with finite products and $f\colon \cat C\to \cat D$ a map in $\Fun(BG,\Cat_\infty)$ whose underlying functor preserves finite products. For each subgroup of finite index $H\subseteq G$, the induces commutative square
\[\begin{tikzcd}
	{\cat C^{hG}} & {\cat D^{hG}} \\
	{\cat C^{hH}} & {\cat D^{hH}}
	\arrow["f^{hG}", from=1-1, to=1-2]
	\arrow["\res^G_H"', from=1-1, to=2-1]
	\arrow["\res_H^G", from=1-2, to=2-2]
	\arrow["f^{hH}",from=2-1, to=2-2]
\end{tikzcd}\]
is vertically right adjointable. 
\end{Lem}

\begin{proof}
Recall that the above square is equivalently given by taking $G$-fixed points on the naturality square for the restriction-coinduction unit map: 
\[\begin{tikzcd}
	\cat C & \cat D \\
	{\CoInd_H^G\res_H^G \cat C} & {\CoInd_H^G\res_H^G \cat D.}
	\arrow["f", from=1-1, to=1-2]
	\arrow["\mathrm{counit}"',from=1-1, to=2-1]
	\arrow["\mathrm{counit}", from=1-2, to=2-2]
	\arrow[from=2-1, to=2-2]
\end{tikzcd}\]
It follows from \cite[Corollary 4.7.4.18]{HA} that if this square is underlying vertically right adjointable, then it remains vertically right adjointable after taking homotopy $G$-fixed points. However, vertical right adjointability of this square is equivalent to the statement that $f$ preserves $G/H$-indexed products. 
\end{proof}

\section{Isotropy separation with coefficients}\label{sec-G-spectra}
In this section we use Theorem \ref{thm-pullback} to establish a symmetric monoidal variant of isotropy separation in equivariant homotopy theory, see \Cref{thm:isotropy_separation}. We point out that our \Cref{cor:isotropy_separation-compact} is a generalization of \cite[Theorem 3.11]{Krause20} to any commutative algebra in $G$-spectra. 

\begin{Not}
Fix a finite group $G$ and a possibly empty family $\cat F$ of subgroups of $G$. We denote by $\Sp_{G}/\cat F$ the $\infty$-category of genuine $G$-spectra with geometric isotropy outside $\cat F$: this is the $\infty$-category of local objects in $\Sp_G$ with respect to the collection 
\[
\K_{\cat F}=\{G/H_+ \mid H \in \cat F\}\subseteq \Sp_G^\omega
\]
where we omitted the suspension spectrum from the notation. Equivalently this is the smashing localization of $\Sp_G$ corresponding to the idempotent commutative algebra $\widetilde{E}\cat F$ so that 
\[
\Sp_{G}/\cat F \simeq \Mod_{\Sp_G}(\widetilde{E}\cat F)
\]
and the functor 
\[
\widetilde{E}\cat F \otimes - \colon \Sp_G \to \Mod_{\Sp_G}(\widetilde{E}\cat F)\simeq \Sp_G/\cat F
\]
is $\K_{\cat F}$-localization. 
\end{Not}

\begin{Not}
Fix coefficients $R\in\CAlg(\Sp_G)$ and consider the base-change functor $R\otimes - \colon \Sp_G \to \Mod_{\Sp_G}(R)$ which is an internal left adjoint in $\Mod_{\Sp_G}(\PrL)$; this follows from \cite[Theorem 1.3]{BalmerAmbrogioSander2016} as we are in the rigidly-compactly generated setting. We set
\[
\Mod_{\Sp_G}(R)/\cat F:= \Mod_{\Sp_G}(R) \otimes_{\Sp_G} \Sp_G/\cat F.
\] 
\end{Not}

Let us record two important remarks about this Lurie tensor product. 

\begin{Rem}\label{rem-base-change}
     We claim that there are equivalences 
     \[
     \Mod_{\Sp_G}(R)/\cat F\simeq  \Mod_{\Sp_G/\cat F}(R\otimes\widetilde E\cat F)\simeq \Mod_{\Sp_G}(R\otimes\widetilde E\cat F)
     \]
     giving other descriptions of the Lurie tensor product. The first equivalence uses that formation of module categories commutes with base-change (apply \cite[Theorem 4.8.4.6]{HA} with $\mathcal{C}:=\Sp_G$, $\mathcal{M}:=\Sp_{G}/\cat F$ and $A:=R$), and the $\Sp_G$-module structure on $\Sp_G/\cat F$. The last equivalence uses that any module $M$ over $R \otimes \widetilde{E} \cat F$ in $\Sp_G$ is automatically $\cat F$-local (being a retract of $M \otimes_R (R \otimes \widetilde{E}\cat F) \simeq M \otimes \widetilde{E}\cat F$), and so lives in $\Sp_G/\cat F$. 
\end{Rem}

\begin{Rem}\label{rem-base-change-dualizable}
   The $\infty$-category $\Mod_{\Sp_G}(R)/\cat F$ is dualizable in $\Mod_{\Sp_G}(\PrL)$ and so the base-change functor $R \otimes -\colon \Sp_G \to \Mod_{\Sp_G}(R)/\cat F$ preserves limits and colimits of $\Sp_G$-modules. This follows from \cite[Corollary D.7.7.4]{SAG}.  Indeed $\Sp_G$ is a locally rigid stable monoidal category in the sense of \cite[Definition D.7.4.1]{SAG} and $\Mod_{\Sp_G}(R)/\cat F$ is compactly assembled as it is compactly generated, see~\cite[Example 21.1.2.3]{SAG}.
\end{Rem}

\begin{Not}\label{not-situation}
    Fix a finite group $G$, a family $\cat F$ of subgroups of $G$, and a subgroup $K\subseteq G$ such that $K\notin\cat F$, but for all proper subgroups $K'\subsetneq K$ we have $K'\in\cat F$. By a slight abuse of notation, we write $\cat F\cup K$ for the smallest family of subgroups of $G$ containing $\cat F$ and $K$. Fix $R \in \CAlg(\Sp_G)$ and set
\[
\K_R:= \{R \otimes G/K_+\otimes \widetilde{E}\cat F\}\subseteq (\Mod_{\Sp_G}(R)/\cat F)^\omega.
\]
\end{Not}

After this preparation we can start to connect to our main result. 

\begin{Lem}\label{lem-local-object}
     Suppose we are in the setting of \Cref{not-situation}. Then base-changing along the map $R\otimes\widetilde E\cat F \to R\otimes\widetilde E(\cat F\cup K)$ in $\CAlg(\Sp_G)$ yields a functor 
    \[
    L_{\cat F\cup K}\colon \Mod_{\Sp_G}(R)/\cat F \to \Mod_{\Sp_G}(R)/\cat F\cup K
    \]
    which is a $\K_R$-localization functor.
\end{Lem}
\begin{proof}
We first claim that the functor 
\[
\Sp_G/\cat F \simeq \Mod_{\Sp_G}(\widetilde{E}\cat F)\xrightarrow{\widetilde{E}(\cat F \cup K) \otimes_{\widetilde{E}\cat F}}- \Mod_{\Sp_G}(\widetilde{E}(\cat F \cup K))\simeq\Sp_G/\cat F \cup K
\]
annihilates $\loct{G/K_+ \otimes \widetilde{E}\cat F}$, and so it is a $\K_{\mathbb{S}_G}$-localization. The kernel of this functor is $
\loct{E(\cat F\cup K)_+\otimes\widetilde E\cat F}$, namely those $G$-spectra with geometric isotropy contained in $\cat F\cup K$. To prove the claim we need to verify the equality
\[
\loct{\widetilde E\cat F\otimes E(\cat F\cup K)_+}=\loct{ G/K_+\otimes\widetilde E\cat F}.
\]
The containment from right to left follows from the observation that the map $\widetilde E\cat F\otimes E(\cat F\cup K)_+\otimes G/K_+ \to\widetilde E\cat F\otimes G/K_+$ is an equivalence, as one can verify using geometric fixed points. The other containment follows from the observation that 
\[
E(\cat F\cup K)_+ \in \loct{G/H_+ \mid H \in \cat F\cup K } 
\]
by noting that $G/H_+ \otimes\widetilde{E}\cat F\simeq 0$ unless $H$ is $G$-conjugate to $K$. This prove the claim and so the lemma in the case $R=\mathbb{S}_G$, the sphere $G$-spectrum. 

To get the general claim, we base-change the functor in the previous paragraph along $R \otimes - \colon \Sp_G\to\Mod_{\Sp_G}(R)$ and apply \Cref{prop-base-change-torsion-local} to deduce that
\[
\begin{tikzcd}
\Mod_{\Sp_G}(R)/\cat F \ar[r, "L_{\cat F\cup K}"] & \Mod_{\Sp_G}(R)/\cat F\cup K.  
\end{tikzcd}
\]
is a $\K_{R}$-localization functor. 
\end{proof}

Let $\Phi^K\colon\Sp_G\to \Sp$ denote the symmetric monoidal functor of geometric $K$-fixed points, and let 
\[ \Phi_R^K \colon \Mod_{\Sp_G}(R) \to \Mod(\Phi^K R)\]
be its $R$-linear extension. 
We will recall below how the $\mathbb E_\infty$-ring $\Phi^K(R)$ is acted on by the Weyl group $W_G(K)$, allowing us to form 
\[
\Mod(\Phi^K(R))^{hW_G(K)}\in \CAlg(\mathrm{Pr}^L_{\mathrm{st}})
\]
as in \Cref{def-hG-modA}. For the next result recall \Cref{not-induced}.

\begin{Prop}\label{prop-completion}
   Suppose we are in the setting of \Cref{not-situation}. Then $\Phi_R^K$ factors through some refined variant $\widetilde\Phi_R^K$ which remembers the Weyl-group action, as follows:
\[ 
\begin{tikzcd}[column sep=large]
 \Phi_R^K\colon  \Mod_{\Sp_G}(R) \ar[r, "\widetilde \Phi_R^K"]  & \Mod(\Phi^K(R))^{hW_G(K)} \ar[r, "\res_e^{W_G(K)}"] &  \Mod(\Phi^K(R)).
\end{tikzcd}
\]
Furthermore, the restriction of the refined $R$-linear $K$-geometric fixed points functor
   \[
   \widetilde{\Phi}_R^K \colon \Mod_{\Sp_G}(R)/\cat F \to\Mod(\Phi^K R)^{h W_G(K)}
   \]
   is a $\K_R$-completion functor and satisfies 
   \[ 
   \widetilde{\Phi}_R^K(R\otimes G/K_+\otimes\widetilde E\cat F)\simeq W_G(K) \circledast \Phi^K(R).
   \]
\end{Prop}
\begin{proof}
We start by reviewing some material from \cite[Subsection 6.5]{MNN}, referring the reader there for more details. Write $\mathcal{O}(G)$ for the category of non-empty and transitive $G$-sets and $\mathcal{O}_{\cat F\cup K}(G)\subseteq{\mathcal O}(G)$ for the full subcategory of objects with isotropy in $\cat F \cup K$. There is a functor
\begin{equation}\label{eq:diagram}
\mathcal{O}(G)^{\mathrm{op}}\to \CAlg(\PrL), \quad  G/H\mapsto\Mod_{\Sp_G}(\mathbb D(G/H_+))\simeq\Sp_H
\end{equation}
with transition maps given by (twisted) restrictions.  
Since $*\in\mathcal{O}(G)$ is final, (\ref{eq:diagram})
determines a symmetric monoidal left adjoint
\begin{equation}\label{eq:completion_map}
\Sp_G\to\lim\limits_{G/H\in\mathcal{O}_{\cat F\cup K}(G)}\Sp_H
\end{equation}
which is a $\K_{\cat F\cup K}$-completion functor for $\K_{\cat F \cup K}=\{G/H_+ \mid H \in\cat F \cup K\}\subseteq \Sp_G^\omega$, see~\cite[Theorem 6.27]{MNN}.  Base-changing along $R \otimes \widetilde{E}\cat F \otimes - \colon \Sp_G\to\Mod_{\Sp_G}(R)/{\cat F}$ we obtain a functor 
\begin{equation}\label{completion}
 \Mod_{\Sp_G}(R)/\cat F  \to \Mod_{\Sp_G}(R)/\cat F\otimes_{\Sp_G}\left(\lim\limits_{G/H\in\mathcal{O}_{\cat F\cup K}(G)}\Sp_H\right)
\end{equation}
which, by \Cref{prop-base-change-torsion-local}, is a completion functor with respect to the collection
\[
\{R \otimes \widetilde{E} \cat F \otimes G/H_+ \mid H \in \cat F \cup K \}.
\]
However note that if $H \in \cat F \cup K$ and $H$ is not $G$-conjugate to $K$, then $\widetilde{E}\cat F \otimes G/H_+ \simeq 0$, so the above collection collapses to 
\[
\K_{R}=\{R \otimes \widetilde{E}\cat F \otimes G/K_+\}.
\]
In other words, we are only left to identify the functor (\ref{completion}) with $\widetilde{\Phi}^K_R$. To this end, we have the following string of equivalences:
\begin{align*}
    \Mod_{\Sp_G}(R)/\cat F \otimes_{\Sp_G} (\lim\limits_{G/H\in\mathcal{O}_{\cat F\cup K}(G)}\Sp_H) &\simeq \lim\limits_{G/H \in \mathcal{O}_{\cat F\cup K}} (\Mod_{\Sp_G}(R)/\cat F\otimes_{\Sp_G}\Sp_H) \\
    & \simeq \lim\limits_{G/H \in \mathcal{O}_{\cat F\cup K}} (\Mod_{\Sp_G}(R\otimes \widetilde{E}\cat F) \otimes_{\Sp_G}\Sp_H) \\
    & \simeq \lim\limits_{G/H\in\mathcal{O}_{\cat F\cup K}(G)} \Mod_{\Sp_H}(\mathrm{Res}^G_H(R\otimes\widetilde E{\cat F})).
\end{align*}
using dualizability of $\Mod_{\Sp_G}(R)/\cat F$ for the first equivalence (\Cref{rem-base-change-dualizable}), \Cref{rem-base-change} for the second equivalence, and finally the fact that module categories behaves well under base-changed \cite[Theorem 4.8.4.6]{HA} for the last one. 
We observe that $\mathrm{Res}^G_H(R\otimes\widetilde E{\cat F})$ is contractible for $H\in\cat F$ and equivalent to $\mathrm{Res}^G_H(R)\otimes\widetilde E{\cat P}_K$ for $H$ conjugate to $K$ (here ${\cat P}_K$ denotes the family of proper subgroups of $K$). This easily implies that the above limit is over a functor which is right Kan extended from the full subcategory $BW_G(K)\subseteq{\mathcal O}_{{\cat F}\cup K}(G)$ spanned by $G/K$. Hence the target further simplifies as 
\[ 
\lim\limits_{G/H\in\mathcal{O}_{\cat F\cup K}(G)} \Mod_{\Sp_H}(\mathrm{Res}^G_H(R\otimes\widetilde E{\cat F}))\simeq\lim\limits_{BW_G(K)} \Mod_{\Sp_K}(\mathrm{Res}^G_K(R)\otimes \widetilde E{\mathcal P}_K).
\]
Since the categorical $K$-fixed points yields an equivalence
\[ 
(-)^K\colon\Mod_{\Sp_K}(\mathrm{Res}^G_K(R)\otimes \widetilde E{\mathcal P}_K)\simeq\Mod(\Phi^K (R))
\]
and the $K$-geometric fixed points on $G$ are by definition given as
\[ \Phi^K:=(-)^K\circ(-\otimes\widetilde E\mathcal{P}_K)\circ\mathrm{Res}^G_K,\]
we conclude that $\Phi_R^K$ refines to a functor 
\[
\widetilde{\Phi}^K_R \colon \Mod_{\Sp_G}(R)/\cat F \to \Mod(\Phi^K (R))^{hW_G(K)}
\]
which is a $\K_{R}$-completion functor. For the final claim, we consider the following diagram of adjunctions
\[
\begin{tikzcd}
  \Phi^K_R \colon  \Mod_{\Sp_G}(R)/\cat F \ar[rr, shift left = 0.5 ex, "\widetilde\Phi^K_R"] & & \ar[ll, shift left = 0.5 ex,  hook, "a" ] \Mod(\Phi^K (R))^{hW_G(K)} \ar[rr, shift left = 0.5 ex, "\res_e^{W_G(K)}"] & & \ar[ll, shift left = 0.5 ex, "W_G(K) \circledast (-)"] \Mod(\Phi^K( R)): C_K^G
\end{tikzcd}
\]
where $a$ is fully faithful as $\widetilde{\Phi}_R^K$ is a completion functor. This easily implies that $\widetilde\Phi^K_R\circ C^G_K\simeq W_G(K)\circledast(-)$. We claim that $R \otimes G/K_+ \otimes \widetilde{E}\cat F \simeq C_K^G(\Phi^K (R))$ which will conclude the proof of the proposition since then 
\begin{align*}
\widetilde\Phi^K_R(R\otimes G/K_+\otimes \widetilde E{\cat F}) & \simeq \widetilde\Phi^K_R(C^G_K(\Phi^K (R)))\\
& \simeq W_G(K)\circledast \Phi^K(R).\\
\end{align*}
To prove the claim we look at the following commutative diagram of adjunctions
\[
\begin{tikzcd}[column sep=1.7 cm]
  \Mod_{\Sp_G}(R)/\cat F \arrow[d,"\sim"] \arrow[drr, bend left, "\Phi^K"] & & \\
  \Mod_{\Sp_G}(R \otimes \widetilde{E}\cat F) \arrow[r, shift left=1.2, "\res^G_K"] & \Mod_{\Sp_K}(\res^G_K(R) \otimes \widetilde{E}\cat P_K) \arrow[l, shift left=1.2, "\CoInd_K^G"] \arrow[r, shift left=1.2, "(-)^K","\sim"'] & \Mod(\Phi^K (R)) \arrow[l, shift left=1.2]\arrow[dll, bend left, "C_K^G"]  \\
  \Mod_{\Sp_G}(R)/\cat F \arrow[u,"\sim"] & & 
\end{tikzcd}
\] 
Since $(\res^G_K(R) \otimes \widetilde{E}\cat P_K)^K \simeq \Phi^K (R)$, the commutativity of the above diagram and the projection formula implies that 
\begin{align*}
C_K^G(\Phi^K(R))& \simeq \CoInd_K^G(\res_K^G(R) \otimes \widetilde{E}\cat P_K)\\
 & \simeq R \otimes \widetilde{E}\cat F \otimes \mathbb{D}(G/K_+) 
 \end{align*}
proving the remaining claim.
\end{proof}
We are finally ready to state and prove our symmetric monoidal variant of isotropy separation. 
\begin{Thm}\label{thm:isotropy_separation}
Suppose we are in the setting of \Cref{not-situation}. Then we have a pullback in $\CAlg(\Pr^L_{\mathrm{st}})$
\[
\begin{tikzcd}
   \Mod_{\Sp_G}(R)/\cat F \arrow[dd,"\mathrm{Ind}((\widetilde{\Phi}_R^K)^\omega)"'] \ar[r," L_{\cat F\cup K}"] 
   &   \Mod_{\Sp_G}(R)/\cat F\cup K \arrow[dd]\\
    & \\
    \Ind(\Perf(\Phi^K(R))^{hW_G(K)})\arrow[r, two heads]& \frac{\Ind(\Perf(\Phi^K(R))^{hW_G(K)})}{\loct{ W_G(K)\circledast \Phi^K(R)}}.
\end{tikzcd}
\]
where the bottom functor is the quotient functor and the right vertical functor is induced by $\mathrm{Ind}((\widetilde{\Phi}_R^K)^\omega)$ and the universal property of $L_{\cat F \cup K}$.
\end{Thm}

\begin{proof}
    We apply \Cref{thm-pullback} to $\CC:=\Mod_{\Sp_G}(R)/\cat F$ and $\K:=\K_R$. The top horizontal arrow in the above diagram agrees with the $\K_R$-localization functor by \Cref{lem-local-object}. Also we can identify the left vertical arrow in the above square using \Cref{prop-completion}. Finally, we can identify the bottom right vertex of the square using the computational statement in \Cref{prop-completion} and the fact that the category of local objects identifies with the Verdier quotient of $ \Ind(\Perf(\Phi^K(R))^{hW_G(K)})$ by the torsion objects.
\end{proof}
We also deduce two ``small'' versions of the previous pullback square. 

\begin{Cor}\label{cor:isotropy_separation-compact}
Suppose we are in the setting of \Cref{not-situation}. Then we have a pullback in $2$-rings:
\[
\begin{tikzcd}
    (\Mod_{\Sp_G}(R)/\cat F)^\omega \arrow[d,"(\widetilde{\Phi}_R^K)^\omega"'] \ar[r,"L^\omega_{\cat F\cup K}"] 
   &   (\Mod_{\Sp_G}(R)/\cat F\cup K )^\omega\arrow[d]\\
    \Perf(\Phi^K(R))^{hW_G(K)}\arrow[r]& (\frac{\Perf(\Phi^K(R))^{hW_G(K)}}{\thickt{W_G(K)\circledast \Phi^K(R)}})^{\natural}
\end{tikzcd}
\]
where unlabeled functors are obtained from those of \Cref{thm:isotropy_separation} by passing to compact objects. 
\end{Cor}

\begin{proof}
    This follows from \Cref{thm:isotropy_separation} by passing to compact objects (which is this case they agree with the dualizable objects).
\end{proof}

\begin{Cor}\label{cor:isotropy_separation-compact}
Suppose we are in the setting of \Cref{not-situation}. Then we have a pullback of symmetric monoidal stable $\infty$-categories:
\[
\begin{tikzcd}
    (\Mod_{\Sp_G}(R)/\cat F)^\omega \arrow[d,"(\widetilde{\Phi}_R^K)^\omega"'] \ar[r,"L^\omega_{\cat F\cup K}"] 
   &   (\Mod_{\Sp_G}(R)/\cat F\cup K )^\omega\arrow[d]\\
    \Perf(\Phi^K(R))^{hW_G(K)}\arrow[r]& \frac{\Perf(\Phi^K(R))^{hW_G(K)}}{\thickt{W_G(K)\circledast \Phi^K(R)}}.
\end{tikzcd}
\]
\end{Cor}

\begin{proof}
    Recall that $\ker(L_{\cat F\cup K}^\omega)=\thickt{R \otimes G/K_+ \otimes  \widetilde{E}\cat F}$, and so by \Cref{prop-completion} we get a factorization 
       \[
    \begin{tikzcd}
        (\Mod_{\Sp_G}(R)/\cat F)^\omega \arrow[d,"(\widetilde{\Phi}_R^K)^\omega"'] \ar[r,"L^\omega_{\cat F\cup K}"] 
   &   (\Mod_{\Sp_G}(R)/\cat F\cup K )^\omega\arrow[d] \arrow[ddr, bend left]& \\
    \Perf(\Phi^K(R))^{hW_G(K)}\arrow[r]\arrow[drr, bend right]& \frac{\Perf(\Phi^K(R))^{hW_G(K)}}{\thickt{W_G(K)\circledast \Phi^K(R)}} \arrow[dr, hook]& \\
     & & (\frac{\Perf(\Phi^K(R))^{hW_G(K)}}{\thickt{W_G(K)\circledast \Phi^K(R)}})^{\natural}
    \end{tikzcd}
    \]
  where the outer square is a pullback. Since the tilted arrow is fully faithful, we see that the pullback of the small square agrees with the pullback of the outer square.
\end{proof}

\bibliographystyle{alpha}
\bibliography{reference}

\end{document}